\documentclass[11pt]{article}

\usepackage{graphicx}%
\usepackage{multirow}%
\usepackage{amsmath,amssymb,amsfonts}%
\usepackage{amsthm}%
\usepackage{mathrsfs}%
\usepackage[title]{appendix}%
\usepackage{xcolor}%
\usepackage{textcomp}%
\usepackage{manyfoot}%
\usepackage{booktabs}%
\usepackage{algorithm}%
\usepackage{algorithmicx}%
\usepackage{algpseudocode}%
\usepackage{listings}%
\usepackage[authoryear]{natbib}

\usepackage[utf8]{inputenc}
\usepackage[T1]{fontenc}
\usepackage{hyperref}
\usepackage{enumitem}
\usepackage{caption}
\usepackage{subcaption}
\usepackage{mathtools}          
\usepackage[capitalize]{cleveref}
\usepackage{authblk}

\newcommand{\bfD}{\boldsymbol D}
\newcommand{\uvec}{\mathbf{ u}}
\newcommand{\vvec}{\mathbf{ v}}
\newcommand{\nuk}{\nu_{k}}
\newcommand{\mcQ}{\mathcal{Q}}
\newcommand{\mcV}{\mathcal{V}}
\newcommand{\pr}{\ensuremath{\pi}} 
\newcommand{\ppow}{\ensuremath{p}} 
\DeclareMathOperator{\Tr}{Tr}

\newtheorem{mylemma}{Lemma}

\newtheorem{myprop}{Proposition}
\newtheorem{myremark}{Remark}
\newtheorem{myassumption}{Assumption}

\newcommand\subfigwidthdoublecolumn{0.49\textwidth}

\newcommand\experimentname{Arolla}

\newlist{enumthm}{enumerate}{1} 
\setlist[enumthm]{label=\upshape(\roman*),
  ref=\upshape\themythm(\roman*)}
\newlist{enumlemma}{enumerate}{1} 
\setlist[enumlemma]{label=\upshape(\roman*),
  ref=\upshape\themylemma(\roman*)}
\newlist{enumprop}{enumerate}{1} 
\setlist[enumprop]{label=\upshape(\roman*),
  ref=\upshape\themyprop(\roman*)}

\crefname{mythm}{Theorem}{Theorems}
\crefname{enumthmi}{Theorem}{Theorems}
\crefname{mylemma}{Lemma}{Lemmas}
\crefname{enumlemmai}{Lemma}{Lemmas}
\crefname{myprop}{Proposition}{Propositions}
\crefname{enumpropi}{Proposition}{Propositions}

\begin{document}

\title{Theoretical results on a block preconditioner used in ice-sheet modeling: eigenvalue bounds for singular power-law fluids}

\author[1]{Christian Helanow \thanks{Corresponding author: christian.helanow@math.su.se}}
\author[1]{Josefin Ahlkrona \thanks{ahlkrona@math.su.se}}
\affil[1]{Dept. Mathematics, Stockholm University, Stockholm, 106 91, Sweden}
\date{}

\maketitle

\begin{abstract}
The properties of a block preconditioner that has been successfully used in finite element simulations of large scale ice-sheet flow is examined. The type of preconditioner, based on approximating the Schur complement with the mass matrix scaled by the variable viscosity, is well-known in the context of Stokes flow and has previously been analyzed for other types of non-Newtonian fluids. We adapt the theory to hold for the regularized constitutive (power-law) equation for ice and derive eigenvalue bounds of the preconditioned system for both Picard and Newton linearization using \emph{inf-sup} stable finite elements. The eigenvalue bounds show that viscosity-scaled preconditioning clusters the eigenvalues well with only a weak dependence on the regularization parameter, while the eigenvalue bounds for the traditional non-viscosity-scaled mass-matrix preconditioner are very sensitive to the same regularization parameter. The results are verified numerically in two experiments using a manufactured solution with low regularity and a simulation of glacier flow. The numerical results further show that the computed eigenvalue bounds for the viscosity-scaled preconditioner are nearly independent of the regularization parameter. Experiments are performed using both Taylor-Hood and MINI elements, which are the common choices for \emph{inf-sup} stable elements in ice-sheet models. Both elements conform well to the theoretical eigenvalue bounds, with MINI elements being more sensitive to the quality of the meshes used in glacier simulations.
\end{abstract}



\section{Introduction}\label{sec:intro}
Glacial ice, as a non-Newtonian fluid, is characterized by having a variable viscosity that is dependent on the shear rate, that is the viscosity is a function of the fluid deformation. Numerical simulation of non-Newtonian fluids is challenging since the varying viscosity makes the equations governing the flow non-linear and the problem possibly ill-conditioned \cite{Leng2012}. On the discrete level the non-linearity introduced by the solution-dependent viscosity is commonly handled by linearizing the problem using the Picard or Newton method. In each Picard or Newton step a linear system has to be solved which, in the case of using iterative solvers, typically requires special preconditioning methods. 

This paper focuses on modeling the non-Newtonian flow of ice, which most commonly is considered to be a gravity-driven flow of a very viscous shear-thinning power-law fluid. We study the preconditioning of the Stokes equations, which govern the steady  creeping flow of ice, using a constitutive equation for which the viscosity, $\nu$, follows the relation 
\begin{equation}\label{eq:viscointro}
 \nu = \nu_0 (\varepsilon^2+ |\mathbf{Du}|^2)^\frac{\ppow-2}{2}.
\end{equation}
Here $\uvec$ is the velocity, $\bfD \uvec=\frac{1}{2}(\nabla \uvec + \nabla \uvec^\top)$ is the strain-rate tensor, $|\mathbf{Du}|^2 = \mathbf{Du}:\mathbf{Du} = D_{ij}D_{ij}$, $\nu_0$  and $\ppow$ are scalar material parameters, and $\varepsilon$ is a regularization parameter.

In the case of glacial ice, standard values for the material parameters are $\ppow=4/3$ and $\varepsilon = 0$ \citep[e.g.,][]{Glen1955, Cuffey2010}. For such shear-thinning flow ($\ppow<2$), we have that $\nu \rightarrow \infty$ when $|\bfD(\uvec)| \rightarrow 0$, making the Stokes equations a singular power-law system \cite{Hirnpowerlaw}.

It is important to improve the numerical methods for modeling the ice sheets on Greenland and Antarctica as such models are important tools in improving our understanding of the consequences of climate change \citep{IPCC2019,tc-14-3033-2020}. For such large-scale simulations direct solvers, which sometimes have been used for both small and large problems \cite[e.g.,][]{Ruckamp_etal2022,Seddik2012}, become memory inefficient and iterative solvers are preferable. Recently various preconditioned iterative solvers have been used in the context of ice-sheet modeling, such as domain decomposition using an overlapping additive Schwarz method, Vanka algorithm or ILU \cite{Ruckamp_etal2022,Zhang2011}. A viscosity-scaled block preconditioner has been used in \citet[e.g.,][]{Leng2012} and, in particular, in \citet{Schannwell_etal2020} where the preconditioner was successfully applied to realistic ice-sheet problems using the finite element model Elmer/Ice \citep{ElmerDescrip}.

The focus of this study is to investigate this type of block-preconditioner and to derive theoretical bounds for eigenvalues of the preconditioned algebraic system. Such block preconditioners have previously been analyzed theoretically for other types of non-Newtonian flow in \citet{GrinevichOlshanskii2009} (for Picard iterations) and \citet{He_etal2015} (for Newton iterations). In these studies the viscosity was considered to be of the form
\begin{equation}\label{eq:binghamvisco}
\nu_{min} + \nu,
\end{equation}
where $\nu_{min}>0$ is a constant scalar material parameter. For the case $\ppow=1$ this is the regularized Bingham model. Many of the estimates derived in \citep{GrinevichOlshanskii2009,He_etal2015} explicitly depend on $\nu_{min}$ and require a non-zero $\nu_{min}$ so that the viscosity is bounded from below.

We adapt the work of \citet{GrinevichOlshanskii2009} and \citet{He_etal2015} to non-Newtonian fluids with a viscosity of the form \cref{eq:viscointro}, i.e. $\nu_{min} = 0$ in \cref{eq:binghamvisco}, by in part using results from \citet{Hirnpowerlaw} specific to power-law fluids of type \cref{eq:viscointro}. Our results are validated in numerical experiments, both for a low-regularity manufactured solution and for a typical ice-modeling benchmark experiment.
We work in a finite element setting and consider both Picard and Newton linearization.

The paper is structured as follows: \Cref{sec:goveq} presents the problem formulation and the equations governing the flow of ice. \Cref{sec:previous} introduces the viscosity-scaled block preconditioner and summarizes previous related results. \Cref{sec:theoryresults} presents bounds on the condition number of the preconditioned system and compares them to a classical non-viscosity-scaled preconditioner. In \Cref{sec:experiments} the theoretical results are verified numerically with two experiments: a manufactured solution of low regularity \citep{Belenki2012} and a problem from the ice-sheet modeling benchmark test suite ISMIP-HOM \citep{Pattyn2008}. Finally, \cref{sec:summary} presents a summary with discussion and conclusions.

\section{Problem Formulation}
\label{sec:goveq} 
\subsection{Governing Equations}\label{sec:governingeq}
The $\ppow$-Stokes equations with Dirichlet boundary conditions are 
\begin{subequations}
 \begin{align}\label{eq:pstokes}
      -\nabla \pr +\nabla \cdot \mathbf{S}(\mathbf{u})  + \mathbf{f} &= \mathbf{0} \quad \text{ in } \Omega,\\
      \nabla\cdot \mathbf{u} &= 0 \quad \text{ in } \Omega,\\
      \mathbf{u} &= \mathbf{g} \quad \text{ on } \Gamma,
\end{align}
\end{subequations}
where $\mathbf{u}=(u_1,\dots, u_d)$ is the $d$-dimensional velocity field, $\pr$ is the pressure, $\mathbf{f}$ is a body force and $\mathbf{g}$ is the velocity at the boundary which must fulfill $\int_\Gamma \mathbf{g} \cdot \mathbf{n} = 0$ as to not violate the divergence constraint. The domain $\Omega \subset \mathbb{R}^d$ is open and bounded with a boundary $\Gamma$. In order to obtain a unique pressure we require that $\int_\Omega \pr \, \text{d}x = 0$. As opposed to the Stokes equations, the stress tensor $\mathbf{S(u)}$ is non-linear due to the non-linear viscosity. The regularized constitutive equation is given as:
\begin{align}
  \mathbf{S(u)} &= \nu(\mathbf{Du})\mathbf{Du},\label{eq:stress}\\
  \nu &:= \nu(\mathbf{Du}) = \nu_0 (\varepsilon^2+ |\mathbf{Du}|^2)^\frac{\ppow-2}{2}, \ \ppow\in (1, 2].\label{eq:viscosity}
\end{align}
The weak formulation of the $\ppow$-Stokes problem is to find velocity and pressure $(\uvec, \pr) \in \mcV \times \mcQ$ so that
\begin{multline}\label{eq:weakStokes}
( \mathbf{S(u)},\mathbf{Dv})_\Omega-(\nabla \cdot \vvec, \pr)_\Omega-(\nabla \cdot \uvec,q)_\Omega=(\mathbf{f}, \vvec)_\Omega \quad \forall  (\vvec,q) \in \mcV  \times \mathcal{Q},
\end{multline}
where we have used $(\cdot, \cdot)_{\Omega}$ to denote the $L^2$ scalar product, e.g., $(\uvec, \vvec)_{\Omega} = \int_{\Omega}\uvec \cdot \vvec \ dx$. Appropriate spaces for the continuous problem are the Sobolev space $\mcV:=[W^{1,\ppow}_\mathbf{g}(\Omega)]^d$ and Lebesgue space $\mcQ := L^{\ppow'}_0(\Omega):=\{L^{\ppow'}(\Omega): (\pr ,1)=0\}$, where $\ppow'=\ppow/(\ppow-1)$ \citep[see e.g.,][]{Hirnpowerlaw}. However, if an assumption is made that $\mathbf{D}\mathbf{u} \in L^\infty(\Omega)$, we can seek the solution in $\mcV:=[H^1_\mathbf{g}(\Omega)]^d$ and $\mcQ := L^2_0(\Omega)$. This amounts to for the discrete problem, resulting from linearizing \cref{eq:weakStokes}, having all iterates $\mathbf{D}\mathbf{u}^k \in L^\infty(\Omega)$ and considering each linear step as a Stokes problem with spatially variable viscosity.

\begin{myremark}
  When considering the flow of ice, the body force in \cref{eq:pstokes} is $\mathbf{f}=\rho \mathbf{\hat{g}}$, where $\rho$ is the density of ice and
  $\mathbf{\hat{g}}$ is the gravitational acceleration. The constitutive equation
for ice, called Glen's flow law \citep{Glen1955}, is typically
expressed as (cf. \cref{eq:stress}):
\[
  \mathbf{S} = 2\nu_0 D_{II}^{\frac{\ppow-2}{2}}\mathbf{Du},
\]
where $D_{II} = \frac{1}{2}\mathbf{D}_{ij}\mathbf{D}_{ij} = \frac{1}{2}\Tr(\mathbf{D^2u})$ is the second invariant of the strain-rate tensor, $\mathbf{Du}$.
Furthermore, stress-free (natural) boundary conditions and no-slip or slip conditions apply at the ice/atmosphere and ice/bedrock interface, respectively, see \cref{fig:Arolla}. These conditions are explained in more detail in  \cref{sec:experiments}. 
\end{myremark} 

\subsection{Linearization}
\label{sec:linearization}
The problem is solved iteratively by using either the Picard or Newton method to linearize \cref{eq:weakStokes}, resulting in a weak formulation where the viscosity depends on a known velocity from either a guess or a velocity solved for in the previous iteration. These methods can be seen as the linearization of the continuous problem. In particular, the non-linear form
\begin{equation}
  a(\uvec)(\vvec):= ( \mathbf{S(u)},\mathbf{D\vvec})_\Omega = \int_\Omega \nu_0 \left( \varepsilon^2 + |\bfD \uvec |^2 \right)^{\frac{\ppow-2}{2}} \bfD \uvec : \bfD \vvec \, d\mathbf{x} 
\end{equation}
which appears in \cref{eq:weakStokes} is linearized as
\begin{multline}
  \label{eq:gateaux}
  a'(\uvec^k)(\delta\uvec, \vvec) =
  \int_\Omega \nu_0 \left( \varepsilon^2 + |\bfD \uvec^k |^2 \right)^{\frac{\ppow-2}{2}} \bfD \delta\uvec : \bfD \vvec \, d\mathbf{x} \\
  + \gamma (\ppow-2)  \int_\Omega \nu_0 \left( \varepsilon^2 + |\bfD \uvec^k |^2 \right)^{\frac{\ppow-4}{2}}
  ( \bfD \uvec^k : \bfD \delta\uvec) (\bfD \uvec^k : \bfD \vvec) \, d\mathbf{x},
\end{multline}
where $\uvec^k$ is the velocity at the $k$:th non-linear iteration, known from the previous iteration or an initial guess, and $\gamma \in \{0, 1\}$. In the linearized problem the update to the velocity $\delta\uvec\in\mcV$ is solved for in each iteration, giving the next velocity iteration as $\uvec^{k+1} = \uvec^k + \delta\uvec$. For $\gamma=1$ \cref{eq:gateaux} equals the G\^{a}teaux derivative of $a(\cdot)(\cdot)$ at $\uvec^k$ in the direction of $\delta\uvec$  \citep[e.g.,][]{Hirnpowerlaw}  which gives the Newton method, while $\gamma=0$ instead results in the Picard method (fixed point iterations).

If we let $b(\uvec, \pr) = -(\nabla \cdot \uvec, \pr)_\Omega$ and $(\uvec^k, \pr^k)$ be a known velocity-pressure pair, the linearized continuous problem reads: in the $k$:th non-linear iteration find the velocity and pressure update $(\delta\uvec,\delta \pr) \in \mcV \times \mcQ$ so that
\begin{multline}
  \label{eq:linearStokes}
  a'(\uvec^k)(\delta\uvec, \vvec) + b(\vvec, \delta \pr) + b(\delta\uvec, q) = \\
    (\mathbf{f}, \vvec)_\Omega - a(\uvec^k, \vvec) - b(\vvec, \pr^k) - b(\uvec^k, q) 
    \quad \forall  (\vvec,q) \in \mcV  \times \mathcal{Q}.
\end{multline}
The updated solution to be used in the $k + 1$ iteration is then set to be \mbox{$(\uvec^{k+1}, \pr^{k+1}) = (\uvec^{k} + \delta\uvec, \pr^{k} + \delta \pr)$}.

In the above, if $\bfD\uvec^k\in L^{\infty}(\Omega)$, it is sufficient to used the spaces  $\mcV:=[H^1_\mathbf{g}(\Omega)]^d$ and $\mcQ := L^2_0(\Omega)$, instead of the spaces specified for the continuous non-linear problem in \cref{eq:weakStokes}. In the case of the discretized problem, this is valid as long as a (discrete) solution for $\uvec^k$ has been found in the previous iteration.

\subsection{Discretization}\label{eq:discretization}
\Cref{eq:linearStokes} is discretized by triangulating the domain $\Omega$ and seeking a discrete solution $(\uvec_h, \pr_h) \in V_h \times Q_h$, where $V_h$ and $Q_h$ are a pair of finite-dimensional spaces. We in this study employ conforming finite elements, i.e., $V_h \subset \mcV$ and $Q_h \subset \mcQ$, that satisfy the so-called \emph{inf-sup} or LBB (Ladyzhenskaya-Babu\v{s}ka-Brezzi) condition \citep{Babuska1973,Brezzi1974}:
\begin{equation}
  \label{eq:lbb}
  c_0 \leq \underset{p_h \in Q_h}{\text{inf}} \underset{\mathbf{u}_h \in V_h}{\text{sup}} \frac{(\nabla \cdot \mathbf{u}_h, \pr_h)}{\|\mathbf{u}_h\|_{V_h} \|\pr_h\|_{Q_h}},
\end{equation}
where $c_0$ is a positive scalar. Specifically we choose  $V_h$, $Q_h$ to correspond to Taylor-Hood $P2P1$ \citep{TaylorHood1974} or MINI elements \citep{Baiocchi1993}, as these are \emph{inf-sup} stable elements that are used in the numerical ice-sheet model Elmer/Ice \citep{Schannwell_etal2020,Gagliardini2013}.

For a unique weak solution to exist to the continuous problem \cref{eq:weakStokes}, the \emph{inf-sup} condition that needs to be satisfied involves the norms defined by the Sobolev spaces $W^{1,\ppow}$ and $L^{\ppow'}$ and monotonicity properties of the constitutive equation \cref{eq:viscosity} \citep{Hirnpowerlaw}. However, with the assumptions on the strain-rate tensor, we here follow \citep{GrinevichOlshanskii2009,He_etal2015} and use the norms induced by the inner products on the Hilbert spaces $H^1_0$ and $L^2$ with the motivation provided in \cref{sec:linearization}.

Let $\{ \varphi_i \}_{i=1}^{n}$ and $\{ \psi_i \}_{i=1}^{m}$ be the bases for $V_h$ and $Q_h$, respectively, and let $u_h = (u_1, \dots, u_{n})^\top\in\mathbb{R}^n$ and $\pr_h = (\pr_1, \dots, \pr_{m})^\top\in\mathbb{R}^m$ represent the discrete values such that $\delta \uvec_h = \sum_{i=1}^{n} u_i \varphi_i$ and $\delta \pr_h = \sum_{i=1}^{m} \pr_i \psi_i$, respectively. We can then formulate the discrete system in each iteration as
\begin{equation}\label{eq:linearsystem}
\begin{bmatrix}
A& B^\top \\
B & 0
\end{bmatrix}\begin{bmatrix}
u\\
\pr
\end{bmatrix}=\begin{bmatrix}
F\\
0
\end{bmatrix},
\end{equation}
with
\begin{align}
  \label{eq:Amatrix}
 A_{i,j}&:=a'(\uvec^k_h)(\varphi_j,\varphi_i)\\
 B_{i,j}&:= -(\nabla \cdot \varphi_i, \psi_i),
\end{align}
where $F$ is the discrete representation of the right-hand side in \cref{eq:linearStokes} and $a'$ depends on a known velocity $\uvec^k_h\in V_h$ and is understood to be defined using $\gamma = 0$ for the Picard method and $\gamma = 1$ for the Newton method, see \cref{eq:gateaux}.
%

\section{Preconditioning}
\label{sec:previous}
Several types of preconditioners for non-Newtonian, variable-viscosity Stokes flow exist in the literature with, in particular, a focus on geodynamics \citep[e.g.][]{May_etal2015,Rudi_etal2017,Fraters_etal2019,Shih_etal2021} and some directly related to ice-sheet models \citep{Isaac_etal2015}.

In this study we choose, in part motivated by the successful but heuristic application in the ice-sheet model Elmer/Ice, to precondition the linear system \cref{eq:linearsystem} with a left block preconditioner of the form
\begin{equation}\label{eq:blockpreconditioner}
P = \begin{bmatrix}
\tilde{A}& 0 \\
B & -\tilde{S}
\end{bmatrix}
\end{equation}
where $\tilde{A}$ is an approximation to $A$ and $\tilde{S}$ is an approximation to the Schur complement $S=B A^{-1} B^\top$. The approximation $\tilde{A}$ can be found with an inexact solver, while finding an appropriate $\tilde{S}$ is more intricate given that $S$ is implicitly defined by $A^{-1}$ and $B$ and cannot be expected to be a sparse matrix. A good choice for $\tilde{S}$ has the property that the eigenvalues of $\tilde{S}^{-1} S$ are well clustered. For a constant viscosity (Newtonian), linear problem, the pressure mass matrix $M$, defined as
\begin{equation}
  \label{eq:Mmatrix}
 M_{i,j} := \left( \psi_i,\psi_j \right),
\end{equation}
is a good choice for $\tilde{S}$ as it is spectrally equivalent to the Schur complement \citep[e.g.][]{Elman_etal2005}.

For variable-viscosity power-law fluids, the viscosity-scaled mass matrix $M_\nu$, defined as 
\begin{equation}
  \label{eq:Mnumatrix}
 M_{\nu,i,j} := \left( \frac{1}{\nu} \psi_i,\psi_j \right),
\end{equation}
is a more appropriate choice for $\tilde{S}$. This type of block preconditioner is used in Elmer/Ice for the shear-thinning power-law flow of ice sheets, but as part of a right preconditioner, i.e., $
\begin{bsmallmatrix*}
  \tilde{A} & B^\top \\
  0 & -M_\nu
\end{bsmallmatrix*}
$ \citep{Malinen_etal2012,manual:Elmer}.
This block preconditioner $\tilde{S}=M$ is, together with the classical choice $\tilde{S}=M$, analyzed in \citet{GrinevichOlshanskii2009} (for the Picard method) and \citet{He_etal2015} (for the Newton method) for a viscosity of the form \cref{eq:binghamvisco} assuming $\nu_{min}>0$. 

Essential to both studies to derive a lower bound for the eigenvalues of $M_{\nu}^{-1}S$ is the following auxiliary \emph{inf-sup} condition
\begin{equation}
\label{eq:nu_infsup}
c_\nu \leq \underset{\pr \in L_{\nu}}{\text{inf}} \underset{\mathbf{u} \in H^1_0}{\text{sup}} \frac{(\nabla \cdot \mathbf{u}, \pr)}{\|\nu^{1/2} \mathbf{Du}\| \|\nu^{-1/2} \pr\|},
\end{equation}
where $c_{\nu}$ is a mesh-independent constant that only weakly depends on the regular \emph{inf-sup} constant $c_0$ and $\nu$, and indirectly on $\varepsilon$ through the dependence of $\nu$. The condition is proven in \citet{GrinevichOlshanskii2009} for $\pr \in L^2_{\nu} := \{\pr \in L^2 : (\pr, \nu^{-1}) = 0\}$ with the additional condition that $(\pr, \nu^{-1/2}) = 0$. Using this \emph{inf-sup} condition it possible to improve the lower bound for the smallest eigenvalue of $M_{\nu}^{-1}S$: without \cref{eq:nu_infsup} (using the LBB condition in \cref{eq:lbb}) the lower bound depends directly on the $\varepsilon$. 

Although the conditions under which \cref{eq:nu_infsup} is valid may not be the most general, theoretical and numerical results from \citet{GrinevichOlshanskii2009} indicate that $M_{\nu}$ should have better preconditioning properties than implied by the directly $\varepsilon$-dependent lower eigenvalue bound found using \cref{eq:lbb}.

The eigenvalue bounds for $\tilde{S}^{-1} S=M^{-1} S$ depend on the material parameter $\nu_{min}$ as
\begin{equation}
\frac{c_0^2}{\nu_{max}} \leq \lambda \leq \frac{1}{\nu_{min}},
\end{equation}
for both Picard and Newton iterations, and for $\tilde{S}^{-1} S=M_{\nu}^{-1} S$ as
\begin{equation}
  c_{\nu}^2 \leq \lambda \leq d \quad\textnormal{and}\quad c_{\nu}^2 \leq \lambda \leq \frac{\nu_{max}}{\nu_{min}}
\end{equation}
for Picard and Newton iterations, respectively \citep{GrinevichOlshanskii2009,He_etal2015}.

These bounds cannot be directly be used for shear-thinning power laws with a viscosity of the form \cref{eq:viscosity} since there is explicit dependence on $\nu_{min}$. The main prospect of this paper is therefore to derive bounds for such power-law fluids. In particular, we are interested in how these bounds depend on the regularization parameter $\varepsilon$ as it gets small, as the constitutive equation conventionally used for the creep of ice amounts to $\varepsilon = 0$ in \cref{eq:viscosity} \citep{Glen1955,Duval1977,Duval1983}.

\section{Theoretical Results}
\label{sec:theoryresults}

In this section we show bounds for the eigenvalues of $\tilde{S}^{-1}S$ for both the classical choice $\tilde{S} = M$, as well as for $\tilde{S} = M_\nu$. The bounds depend on the \emph{inf-sup} constant $c_0$ (see \cref{eq:lbb}) and on the constant $c_\nu$ (see \cref{eq:nu_infsup}). These constants are determined numerically using the methods described in \cref{sec:infsupbounds}.

Since we are considering $\uvec\in H^1_0$ (pure Dirichlet boundary conditions) the Schur complement $S$ has a kernel consisting of the constant vector corresponding to the constant pressure mode, which results in a zero eigenvalue. However, since this is in practice circumvented numerically by either setting an additional constraint on the pressure (e.g., $\int_{\Omega} \pr \ dx = 0$) or by setting a null space for the linear system, we below treat $S$ as non-singular and $B$ as full rank.

To facilitate the structure of the proof of \cref{prop:mainprop}, we first list an assumption and some useful inequalities.

\begin{myassumption}
  \label{ass:strainrate}
  The following assumption is made on the deformation tensor:
  \[
    \mathbf{D}\mathbf{u}^k\in L^\infty(\Omega).
  \]
\end{myassumption}

\begin{mylemma}\label{lem:A} The following inequalities involving the strain-rate tensor hold:
  \begin{enumlemma}
  \item $\|\nabla\cdot\uvec\| \leq \|\mathbf{Du}\| \leq \|\nabla\uvec\| \quad\forall \uvec\in H^1_0(\Omega)$ \label{lem:A1}
  \item $\|\nu^{1/2}\nabla\cdot\uvec\| \leq \sqrt{d}\|\nu^{1/2}\mathbf{Du}\| \quad\forall \uvec\in H^1(\Omega)$, \\
    where $d$ is the dimension. \label{lem:A2}
  \end{enumlemma}
\begin{proof}
  (i) See e.g., \citet{GrinevichOlshanskii2009,John2016}\\
  (ii) See \citet{Kaiser2014}.\\
\end{proof}

\end{mylemma}

\begin{mylemma}\label{lem:B} For the viscous term of the G\^{a}teaux derivative, \cref{eq:gateaux}, the following hold:
  \begin{enumlemma}
  \item $a'(\uvec^k)(\uvec, \uvec) \geq (1 + \gamma(\ppow-2))\|\nuk^{1/2}\mathbf{Du}\|^2$
    and\\
    $a'(\uvec^k)(\uvec, \uvec) \geq (1 + \gamma(\ppow-2)) \nu_0 (\varepsilon^2 + \|\mathbf{Du}^k\|^2_\infty)^{\frac{\ppow-2}{2}}\|\mathbf{Du}\|^2,$
    $\forall\ \uvec\in H^1_0(\Omega)$, where $\nuk = \nu(\mathbf{D\uvec}^k)$.\label{lem:B1}
  \item $a'(\uvec^k)(\uvec, \uvec) \leq \int_{\Omega} \nuk |\mathbf{Du}|^2\ dx = \|\nuk^{1/2}\mathbf{Du}\|^2 \leq \nu_{max} \|\mathbf{Du}\|^2$,\\
    where $\nu_{max} = \nu(\mathbf{0}) = \nu_0 \varepsilon^{\ppow-2}$.\label{lem:B2}
  \end{enumlemma}
\begin{proof} The outline of this proof can be deduced from adding intermediate steps to a proof from \citet{Hirnpowerlaw}.\\
  (i) Noting that for $\gamma\geq 0$ and $\ppow \leq 2$ the second term in \cref{eq:gateaux} is always non-positive, applying Cauchy-Schwartz inequality and using that $ |\bfD \uvec^k |^2 < \left( \varepsilon^2 + |\bfD \uvec^k |^2 \right)$, we have that
  \begin{align*}
    a'(\uvec^k)(\uvec, \uvec) &= \int_\Omega \nu_0  \left( \varepsilon^2 + |\bfD \uvec^k |^2 \right)^{\frac{\ppow-2}{2}} (\bfD \uvec : \bfD \uvec) \, d\mathbf{x} \\
     &\quad + \gamma (\ppow-2)  \int_\Omega \nu_0 \left( \varepsilon^2 + |\bfD \uvec^k |^2 \right)^{\frac{\ppow-4}{2}}|\bfD \uvec^k : \bfD \uvec|^2 \, d\mathbf{x}\\
    &\geq \int_\Omega \nu_0  \left( \varepsilon^2 + |\bfD \uvec^k |^2 \right)^{\frac{\ppow-2}{2}} |\bfD \uvec|^2 \, d\mathbf{x} \\
    &\quad + \gamma (\ppow-2)  \int_\Omega \nu_0    \left( \varepsilon^2 + |\bfD \uvec^k |^2 \right)^{\frac{\ppow-4}{2}}|\bfD \uvec^k|^2 |\bfD \uvec|^2 \, d\mathbf{x}\\
    &\geq \int_\Omega \nu_0  \left( \varepsilon^2 + |\bfD \uvec^k |^2 \right)^{\frac{\ppow-2}{2}} |\bfD \uvec|^2 \, d\mathbf{x} \\
    &\quad + \gamma (\ppow-2)  \int_\Omega \nu_0 \left( \varepsilon^2 + |\bfD \uvec^k |^2 \right)^{\frac{\ppow-2}{2}} |\bfD \uvec|^2 \, d\mathbf{x}\\
    &= (1 + \gamma(\ppow-2))\int_\Omega \nu_0  \left( \varepsilon^2 + |\bfD \uvec^k |^2 \right)^{\frac{\ppow-2}{2}} |\bfD \uvec|^2 \, d\mathbf{x}.\\
  \end{align*}
From this, the first inequality follows from using \cref{eq:viscosity} with the definition of $\nuk$ and the second inequality follows from the definition of $\|\cdot\|_{\infty}$.\\
(ii) Again, using that the second term in \cref{eq:gateaux} is non-positive, we have that
\begin{align*}
  a'(\uvec^k)(\uvec, \uvec)  &\leq \int_\Omega \nu_0 \left( \varepsilon^2 + |\bfD \uvec^k |^2 \right)^{\frac{\ppow-2}{2}} |\bfD \uvec|^2 \, d\mathbf{x} \\
                             &\leq\int_\Omega \varepsilon^{p-2} |\bfD \uvec^k |^2 \, d\mathbf{x} = \nu_{max} \|\bfD \uvec\|^2 .
\end{align*}
\end{proof}
\end{mylemma}


\begin{myprop}[Eigenvalue bounds for $\tilde{S}^{-1}S$]\label{prop:mainprop}
  Consider the Stokes problem with power-law fluid \cref{eq:viscosity} linearized using either the Picard ($\gamma=0$) or the Newton ($\gamma=1$) method as in \cref{eq:linearStokes}. Let $\lambda$ denote an eigenvalue of $\tilde{S}^{-1}S$, where $\tilde{S} = M$ or $\tilde{S} = M_\nu$. The following bounds for $\lambda$ hold:
  \begin{enumprop}
  \item \label{prop:M_bounds} For $\tilde{S} = M$:\\
    \[
      c_0^2 \varepsilon^{2-\ppow} \leq \lambda \leq \frac{(\varepsilon^2 + \|\mathbf{Du}_h^k\|^2_\infty)^{\frac{2-\ppow}{2}}}{\nu_0 (1 + \gamma(\ppow-2))}
    \]
  \item \label{prop:Mnu_bounds} For $\tilde{S} = M_\nu$:\\
    \[
      c_\nu^2 \leq \lambda \leq \frac{d}{1 + \gamma(\ppow-2)}
    \]
    where $d = 2, 3$ is the dimension of the problem.
  \end{enumprop}
\end{myprop}

\begin{proof}
  Denoting by $\langle v, w \rangle = w^\top v$ the Euclidean inner product, we have that
  \begin{align*}
    a'(\uvec^k_h)(\vvec_h, \vvec_h) &= \langle{A v, v}\rangle,\\
    \|q_h\|^2 &= \langle{M q, q}\rangle,\\
    \|\nuk^{-1/2} q_h\|^2 &= \langle{M_\nu q, q}\rangle.
  \end{align*}
  Furthermore, due to the definitions of $A$ (symmetric positive definite) and $B$, we have that
  \begin{align}
    \nonumber\label{eq:schur_limit}
    \langle{Sq, q}\rangle
    &= \langle{BA^{-1}B^\top q, q}\rangle = \langle{A^{-1}B^\top q, B^\top q}\rangle = \\
    &= \underset{v \in \mathbb{R}^n}{\text{sup}}{\frac{\langle{v, B^\top q}\rangle^2}{\langle{Av, v}\rangle}} = \underset{\vvec_h \in V_h}{\text{sup}}{\frac{(\nabla\cdot \vvec_h, q_h)^2}{a'(\uvec^k)(\vvec_h, \vvec_h)}}.
  \end{align}
  Proof of (i): From \cref{lem:B1} followed by \cref{lem:A1} and Cauchy-Schwartz inequality we have that:
  \begin{align*}
    \underset{\vvec_h \in V_h}{\text{sup}} {\frac{(\nabla\cdot \vvec_h, q_h)^2}{a'(\uvec^k_h)(\vvec_h, \vvec_h)}} & \leq \underset{\vvec_h \in V_h}{\text{sup}}{\frac{(\varepsilon^2 + \|\mathbf{Du}^k_h\|^2_\infty)^{\frac{2-\ppow}{2}}(\nabla\cdot \vvec_h, q_h)^2}{\nu_0 (1 + \gamma(\ppow-2)) \|\mathbf{Dv}_h\|^2}}\\
                                                                                                                  &\leq \underset{\vvec_h \in V_h}{\text{sup}}{\frac{(\varepsilon^2 + \|\mathbf{Du}^k_h\|^2_\infty)^{\frac{2-\ppow}{2}}(\nabla\cdot \vvec_h, q_h)^2}{\nu_0 (1 + \gamma(\ppow-2)) \|\nabla\cdot\vvec_h\|^2}} \\
                                                                                                                  &\leq \underset{\vvec_h \in V_h}{\text{sup}}{\frac{(\varepsilon^2 + \|\mathbf{Du}^k_h\|^2_\infty)^{\frac{2-\ppow}{2}}\|\nabla\cdot\vvec_h\|^2 \|q_h\|^2}{\nu_0(1 + \gamma(\ppow-2)) \|\nabla\cdot\vvec_h\|^2}}\\
                                                                                                                  &\leq \frac{(\varepsilon^2 + \|\mathbf{Du}^k_h\|^2_\infty)^{\frac{2-\ppow}{2}}}{\nu_0(1 + \gamma(\ppow-2))} \langle Mq_h, q_h \rangle,
  \end{align*}
  which together with \cref{eq:schur_limit} gives the upper bound of (i).
  Similarly the lower bound can be derived by using \cref{lem:B2}, followed by \cref{lem:A1,eq:lbb}:
  \begin{equation*}
    \underset{\vvec_h \in V_h}{\text{sup}}{\frac{(\nabla\cdot \vvec_h, q_h)^2}{a'(\uvec^k_h)(\vvec_h, \vvec_h)}} \geq
    \underset{\vvec_h \in V_h}{\text{sup}}{\frac{(\nabla\cdot \vvec_h, q_h)^2}{\nu_{max} \|\mathbf{Dv}_h\|^2}}\geq\underset{\vvec_h \in V_h}{\text{sup}}{\frac{(\nabla\cdot \vvec_h, q_h)^2}{\nu_{max} \|\nabla\vvec_h\|^2}} \geq \frac{c_0^2}{\varepsilon^{\ppow-2}} \langle Mq, q\rangle.
  \end{equation*}
  Proof of (ii): From \cref{lem:B1}, followed by the Cauchy-Schwartz inequality and \cref{lem:A2} we have that
  \begin{align*}
    \underset{\vvec_h \in V_h}{\text{sup}} {\frac{(\nabla\cdot \vvec_h, q_h)^2}{a'(\uvec^k_h)(\vvec_h, \vvec_h)}} & \leq \underset{\vvec_h \in V_h}{\text{sup}}{\frac{(\nuk^{1/2} \nabla\cdot \vvec_h, \nuk^{-1/2} q_h)^2}{(1 + \gamma(\ppow-2))\|\nuk^{1/2}\mathbf{Dv}_h\|^2}}\\
    &\leq \underset{\vvec_h \in V_h}{\text{sup}}{\frac{\|\nuk^{1/2} \nabla\cdot \vvec_h\|^2  \|\nuk^{-1/2} q_h\|^2}{(1 + \gamma(\ppow-2))\|\nuk^{1/2}\mathbf{Dv}_h\|^2}}\\
                                                                                                                          &\leq \underset{\vvec_h \in V_h}{\text{sup}}{\frac{d \|\nuk^{1/2}\mathbf{Dv}_h\|^2  \|\nuk^{-1/2} q_h\|^2}{(1 + \gamma(\ppow-2))\|\nuk^{1/2}\mathbf{Dv}_h\|^2}}\\
                                                                                                                          & \leq \frac{d}{(1 + \gamma(\ppow-2))} \langle M_{\nu} q, q \rangle,
  \end{align*}
  which together with \cref{eq:schur_limit} gives the upper bound of \emph{(ii)}. 
  The lower bound can be derived using \cref{lem:B2} and the auxiliary \emph{inf-sup} condition in \cref{eq:nu_infsup}
  \begin{equation*}
    \underset{\vvec_h \in V_h}{\text{sup}}{\frac{(\nabla\cdot \vvec_h, q_h)^2}{a'(\uvec^k_h)(\vvec_h, \vvec_h)}} \geq
    \underset{\vvec_h \in V_h}{\text{sup}}{\frac{(\nabla\cdot \vvec_h, q_h)^2}{\|\nuk^{1/2} \mathbf{Dv}_h\|^2}}
    \geq c_{\nu}^2 \langle M_\nu q, q \rangle,
  \end{equation*}
  which, together with \cref{eq:schur_limit}, gives the lower bound of (ii).
\end{proof}

\begin{myremark}
The constants $c_0$ and $c_\nu$ will be determined numerically, see \cref{sec:infsupbounds}. The constant $c_0$ is independent of $\varepsilon$, while theoretical results from \citet{GrinevichOlshanskii2009} imply that $c_\nu$ will only weakly depend on $\varepsilon$.
\end{myremark}

\begin{myremark}
  For the classical preconditioner $\tilde{S}=M$ the lower bound is proportional to power of $\varepsilon$ (given that $\ppow < 2$) and the upper bound depends on $\varepsilon$ and the maximum value of the strain rate in such a way that the bound for the ratio $\lambda_{max}/\lambda_{min}$ increases for as $\varepsilon$ becomes smaller. This is, in general, bad for the performance of linear solvers. The choice  $\tilde{S}=M_\nu$ is clearly better as it has a constant upper bound and the lower bound has, as shown in \citet{GrinevichOlshanskii2009}, a much weaker dependence on $\varepsilon$. The Picard method ($\gamma=0$) results in a slightly lower bound for $\lambda_{max}$, however, the better non-linear convergence properties of the Newton method will for most applications likely make the Newton method the overall faster option.
\end{myremark}

\subsection{Determining the \emph{inf-sup} constants $c_0$ and $c_\nu$}
\label{sec:infsupbounds}
To investigate how the eigenvalue bounds from \cref{sec:theoryresults} relate to numerically computed eigenvalues (\cref{sec:experiments}), we compute the two \emph{inf-sup} constants $c_0$ and $c_\nu$ numerically.

We follow the method presented in \citet{Qin1994,ArnoldRognes2009} and make use of and modify the software ASCoT \citep{Rognes2009} which is Python module built on top of the FEniCS framework \citep{Alnaes2015} that automates the testing of stability conditions like those in \cref{eq:lbb,eq:nu_infsup}. The \emph{inf-sup} is the square root of the minimum eigenvalue for the following generalized eigenvalue problem: find $0 \neq (\uvec_h, \pr_h)\in V_h\times Q_h$ so that 
\begin{equation*}
  \langle \uvec_h , \vvec_h \rangle_V + b( \vvec_h, \pr_h ) + b( \uvec_h , q ) = -\lambda \langle \pr_h , q \rangle_Q,\quad \forall (\vvec_h, q) \in V_h \times Q_h
\end{equation*}
where $\langle \cdot, \cdot \rangle_{V,Q}$ represents the inner product on the velocity and pressure space, respectively, that induces the norms used in the two separate \emph{inf-sup} conditions \cref{eq:lbb,eq:nu_infsup}. For the numerical computation of $c_\nu$, we make the assumption that the above variational formulation, proven in \citet[e.g.,][]{Qin1994} for $c_0$, also holds for $c_\nu$.
The discrete eigenvalue problem is solved using SLEPc
\citep{Hernandez_etal2005,slepc-users-manual} with appropriate
restrictions, such as basis for the nullspace consisting of the
constant functions in the case of pure Dirichlet conditions.

\section{Numerical Experiments}\label{sec:experiments}
In this section we numerically confirm that the bounds of \cref{prop:mainprop} hold. Since the bounds depend on $\varepsilon$, a parameter that should be as small as possible in ice-sheet models, we study how the eigenvalues vary with $\epsilon$. We run experiments using Picard and Newton iterations for both $P2P1$ elements and MINI-elements. All simulations are performed using an in-house Python module built on top of the FEniCS framework \citep{Alnaes2015,fenics:book} compiled with SLEPc \citep{Hernandez_etal2005,slepc-users-manual} as a backend.

To compute the eigenvalues of $\tilde{S}^{-1}S$ we assemble the matrices $\tilde{S}, B, B^\top$ and $A$ and explicitly compute the inverses $A^{-1}$ to assemble $S=BA^{-1}B^\top$. The generalized eigenvalue problem
\begin{equation}
  \label{eq:schur_eigen}
  Sx = \lambda\tilde{S}x
\end{equation}
is then solved using SLEPc. Given the expensive operation of directly computing $A^{-1}$, we only assemble the matrices and solve for the eigenvalues in the last Picard or Newton iteration, i.e., in the iteration for which the method has converged to a given tolerance. The method is considered to have convergence when a tolerance of relative tolerance $r_{tol} = 10^{-6}$ or absolute tolerance $a_{tol} = 10^{-10}$ has been reached. At the time of convergence, the solution $(\uvec, \pr)$ together the in the non-linear iteration computed represent the final state of the system, so we deem this approach to be sufficient for investigating the theoretical bounds. We, however, in a few cases computed the eigenvalues for \cref{eq:schur_eigen} for every iteration in a simulation run: the outcomes of these did not change any of the results compared to the last iteration.

The system in each non-linear iteration is solved with PETSc \cite{BalayEtAl1997,webpage:petsc} using GMRES \citep{SaadSchultz1986} preconditioned by $P$ with either $\tilde{S}=M$ or $\tilde{S}=M_\nu$ (\cref{eq:blockpreconditioner}), with the action of $P^{-1}$ approximated using AMG \citep{HensonYang2002}.

\subsection{Model problems}\label{sec:modelproblems}

\subsubsection{A manufactured solution for power-law fluids}
We use the manufactured solution (MS) presented in \citet{Belenki2012} to solve the Dirichlet problem \cref{eq:pstokes} in two dimensions on the domain $(x,y)\in [-1, 1]\times [-1, 1]$. The right-hand side and Dirichlet boundary conditions are given by inserting the specified solution of the problem,
\begin{equation}\label{eq:belenkisolution}
\uvec=|(x,y)|^{a-1}(x,-y)^\top, \quad \pr=|(x,y)|^b,
\end{equation}
into \cref{eq:pstokes}. For the solution \cref{eq:belenkisolution} to be regular enough it is required that $a > 1$ and $b > -1 + \frac{2}{\ppow}$. The manufactured solution allows us to make sure that 1) our solution is correct, and 2) control the regularity of the problem. Since we want to ensure that our results hold for challenging problems we follow \citep{Belenki2012} and set $a = 1 + \delta$ and $b = -1 + \frac{2}{\ppow} + \delta$ with $\delta = 0.01$, specifying the solution to be of very low regularity. The material parameters are set to $\nu_0 = 1$ and $\ppow = 4/3$.

The square domain is discretized using a structured mesh which subdivides each side in $nx$ sections, where the resulting rectangles are diagonally cut to form $2 (nx)^2 $ triangle elements. The typical mesh in this experiment used $nx = 32$, resulting in 2048 elements. The initial guess starting the non-linear iteration is zero when using the Picard method, while the Newton method uses an initial guess solved for using the Picard method with a low relative tolerance ($r_{tol} = 10^{-2}$) for 5 iterations.

\subsubsection{A glacial-ice benchmark experiment}
To investigate how well the theory applies to practical ice-modeling examples we run the classical benchmark experiment ISMIP-HOM E, which consists of a stationary simulation of Haut Glacier d'Arolla situated in the Swiss alps \citep{Pattyn2008}, see \cref{fig:Arolla}.  
\begin{figure*}[h!]
\centering
      \includegraphics[width=0.7\textwidth]{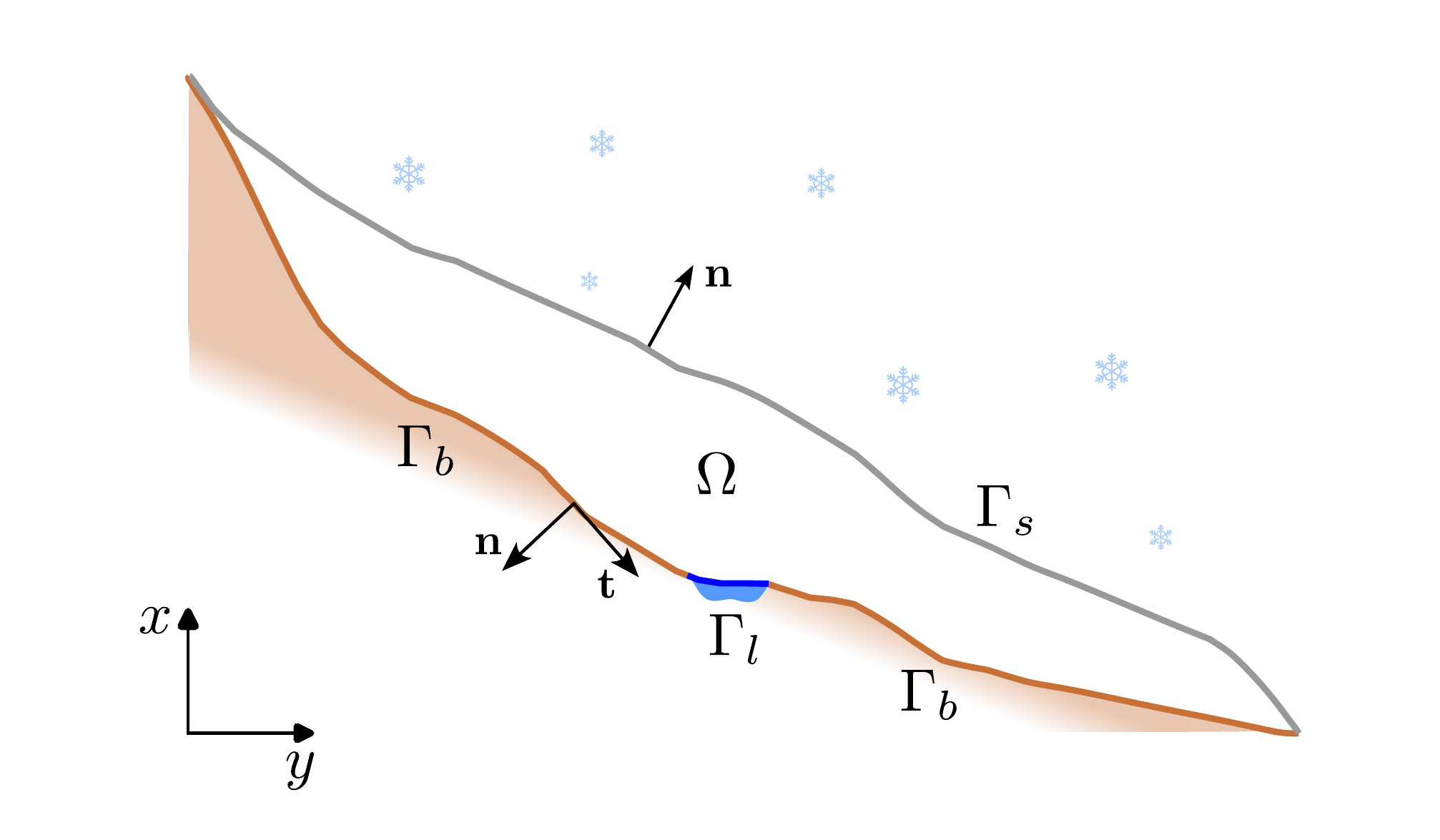}
      \caption{A conceptual cross section of Haut Glacier d'Arolla, the domain used for experiment \experimentname{}. The boundary of the domain $\Gamma$ is subdivided into the glacier surface $\Gamma_s$ (gray line) and the base of the ice which rests partly on bedrock ($\Gamma_b$, brown line) and partly on a subglacial lake ($\Gamma_l$, blue line).} \label{fig:Arolla}
\end{figure*}
The glacier boundary $\Gamma$ consists of the ice/atmosphere interface $\Gamma_s$ and ice/bedrock interface $\Gamma_b$. The ice velocity and pressure is given as the solution to \cref{eq:pstokes} with $\ppow = 4/3$, $\nu_0$ set to a typical value used for simulations of isothermal ice, $\mathbf{f}=\rho \mathbf{\hat{g}}$ and with appropriately modified boundary conditions. Over the subglacial lake, $\Gamma_l$, a Navier-slip condition applies with an impenetrability condition in the normal direction and free slip in the tangential direction, while a no-slip condition applies to the rest of the bed $\Gamma_b$. At the ice surface, $\Gamma_s$, a stress-free condition applies. If we by $\mathbf{n}$ and $\mathbf{t}$ denote the to the boundary outward-pointing unit normal and tangential vectors, respectively, the boundary conditions can be summarized as:

\begin{subequations}
\begin{align}
  \uvec\cdot\mathbf{n} &= 0 \text{ on } \Gamma_l,\label{eq:bc_imp}\\
  \mathbf{t} \cdot \mathbf{S} \cdot \mathbf{n} &= 0 \text{ on } \Gamma_l,\label{eq:bc_freeslip}\\
  (\mathbf{S}-\pr\mathbf{I})\cdot \mathbf{n} &= \mathbf{0} \text{ on } 
                                             \Gamma_s,\label{eq:bc_stressfree}\\
  \uvec &= \mathbf{0} \text{ on } \Gamma_b,\label{eq:bc_noslip}.
\end{align}
\end{subequations}
The impenetrability condition \cref{eq:bc_imp} is here implemented strongly through a local rotation of the coordinate system in the direction of a discrete normal average of neighboring cell facets at the degree of freedom \citep{John2002,manual:Elmer}. The boundary conditions \cref{eq:bc_freeslip,eq:bc_stressfree} are both natural boundary conditions that are weakly implemented through the variational formulation, while \cref{eq:bc_noslip} is a Dirichlet condition on the velocity.

The unstructured mesh used to discretize the Haut Glacier d'Arolla
experiment is the result of a Delauney triangulation of the domain
using \texttt{Gmsh} \citep{Geuzaine2009}, with the characteristic mesh
size, $lc$, being a measure of the typical cell size (longest edge of
a triangle) of the mesh. The typical mesh used in this experiment
consists of 1620 triangles with $lc=32$. The initial guess starting
the non-linear iteration is zero using both the Picard and Newton
method.

\begin{figure}[h!]
  \centering
  \begin{subfigure}[b]{\subfigwidthdoublecolumn}
    \centering
    \includegraphics[width=\textwidth]{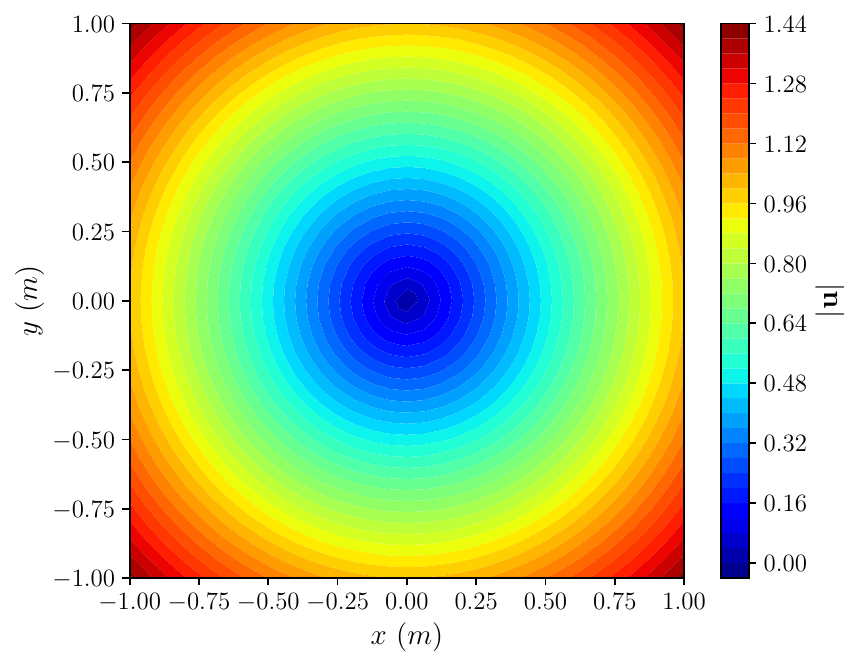}
    \caption{Simulated speed, $|\uvec|$, for the MS with low regularity experiment using $P2P1$}
    \label{fig:simulation_belenki}
  \end{subfigure}
  \hfill
  \begin{subfigure}[b]{\subfigwidthdoublecolumn}
    \centering
    \includegraphics[width=\textwidth]{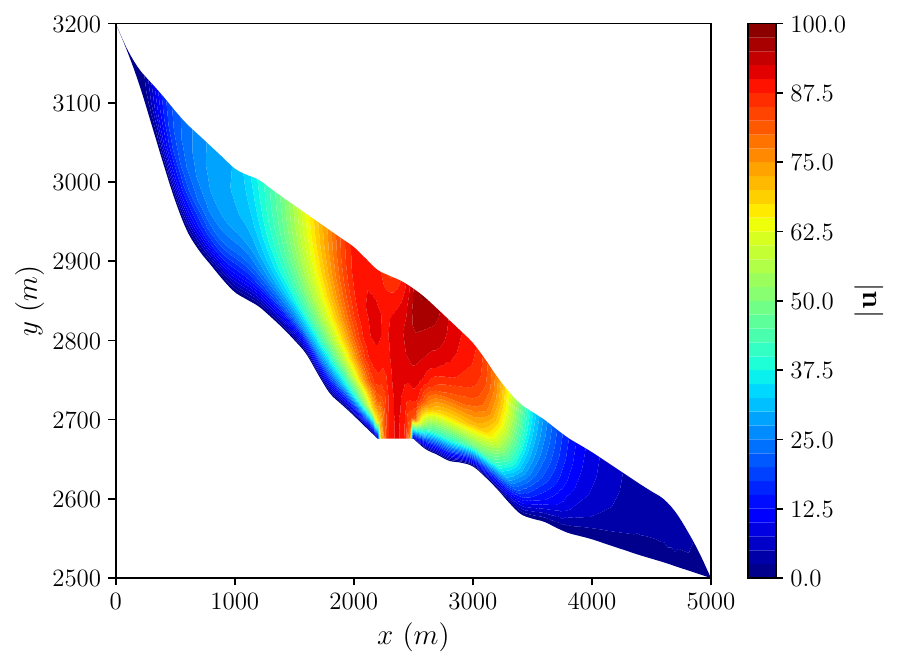}
    \caption{Simulated speed, $|\uvec|$, for the \experimentname{} experiment using $P2P1$}
    \label{fig:simulation_arolla}
  \end{subfigure}
  \caption{Simulation results from the two experiments: manufactured
    solution (MS) with low regularity and Haut Glacier d'Arolla
    (\experimentname{}). The colors shows the magnitude of the velocity,
    $|\uvec|$.}
  \label{fig:simulations}
\end{figure}

\subsection{Results for $\tilde{S}=M$}
\label{sec:results_M}
The theoretical upper and lower eigenvalue bounds, using the numerically computed $c_0$, for $M^{-1}S$ in the final Newton iteration are shown in red dashed and green dotted lines, respectively, in \cref{fig:Mgraphs}, and computed largest and smallest eigenvalues are shown as up and down triangle markers, respectively ($\lambda_{max}$ in red and $\lambda_{min}$ in green). Theory predicts that the lowest eigenvalue $\lambda_{min}$ decreases with $\varepsilon$ while the largest eigenvalue should be bounded from below. This leads to a large ratio $\lambda_{max}/\lambda_{min}$ ($10^4$-$10^6$ for the glacier) which in general leads to bad performance of linear solvers. For the manufactured problem \cref{eq:belenkisolution} theory and experiments align very well for $\lambda_{max}$, while the experimentally computed $\lambda_{min}$ are larger than the theoretical bound predicts for the manufactured problem. We believe this is due to that there is an unresolved very localized peak in the viscosity for the manufactured problem. On Haut Glacier d'Arolla the viscosity is high in a more distributed area around the ice/atmosphere interface, and so theory and experiments align very well for $P2P1$ elements. The MINI elements result in a significantly lower value for $c_0$ for the \experimentname{} domain and resulting lower theoretical bounds for $\lambda_{min}$. The small value for $c_0$ is related to the locally poor quality of elements close in parts of the domain at this mesh resolution (see \cref{sec:mesh_quality}) and most likely leads to small enough eigenvalues that numerical errors of the eigenvalue computations become significant (\cref{fig:Marolla_mini}). 

In addition to the \experimentname{} experiment, we performed simulations with only no-slip (Dirichlet) boundary conditions at the bed, i.e., we set $\Gamma_l = \emptyset$: we found no significant change comparing the results of these simulations to \experimentname{} indicating that the boundary conditions, at least in this case, do not have a significant impact on the character of the preconditioned system, but that the results are rather affected by the value of $c_0$ and the strong dependence on $\varepsilon$. 

\begin{figure*}[h!]
  \centering
  \begin{subfigure}[b]{\subfigwidthdoublecolumn}
    \centering
    \includegraphics[width=\textwidth]{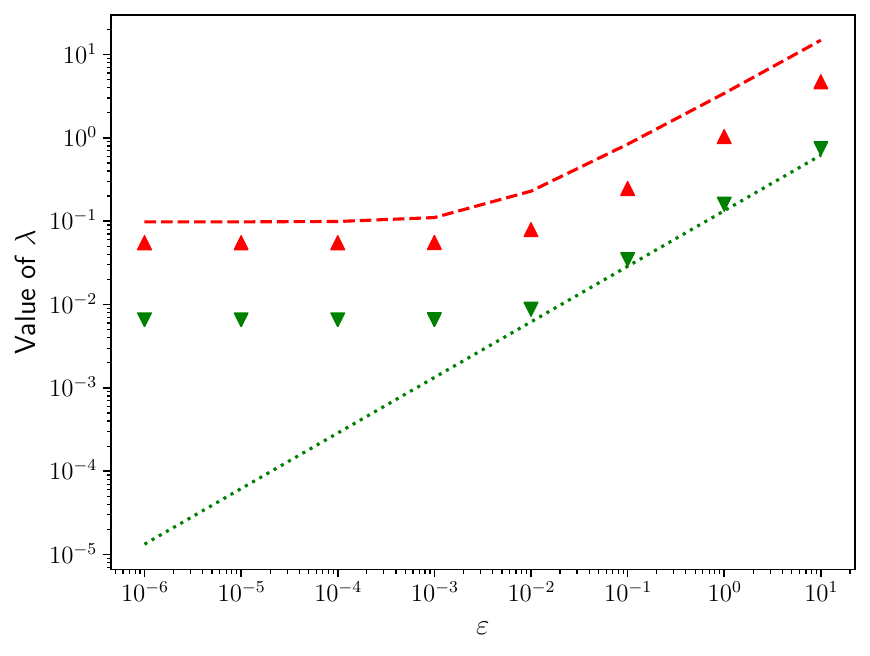}
    \caption{MS with low regularity - $P2P1$}
    \label{fig:Mbelenki}
  \end{subfigure}
  \hfill
  \begin{subfigure}[b]{\subfigwidthdoublecolumn}
    \centering
    \includegraphics[width=\textwidth]{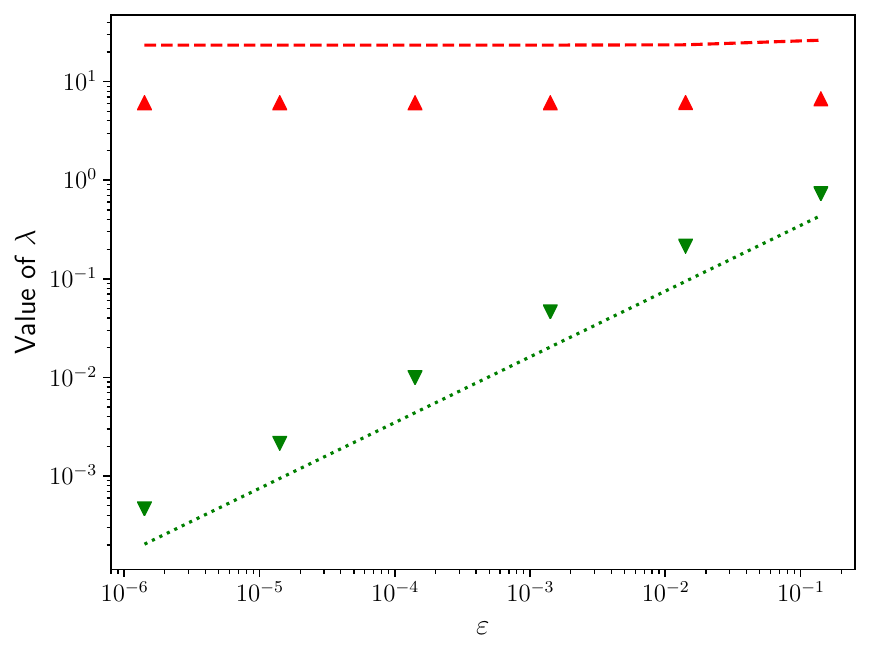}
    \caption{\experimentname{} - $P2P1$}
    \label{fig:Marolla}
  \end{subfigure}
  \begin{subfigure}[b]{\subfigwidthdoublecolumn}
    \centering
    \includegraphics[width=\textwidth]{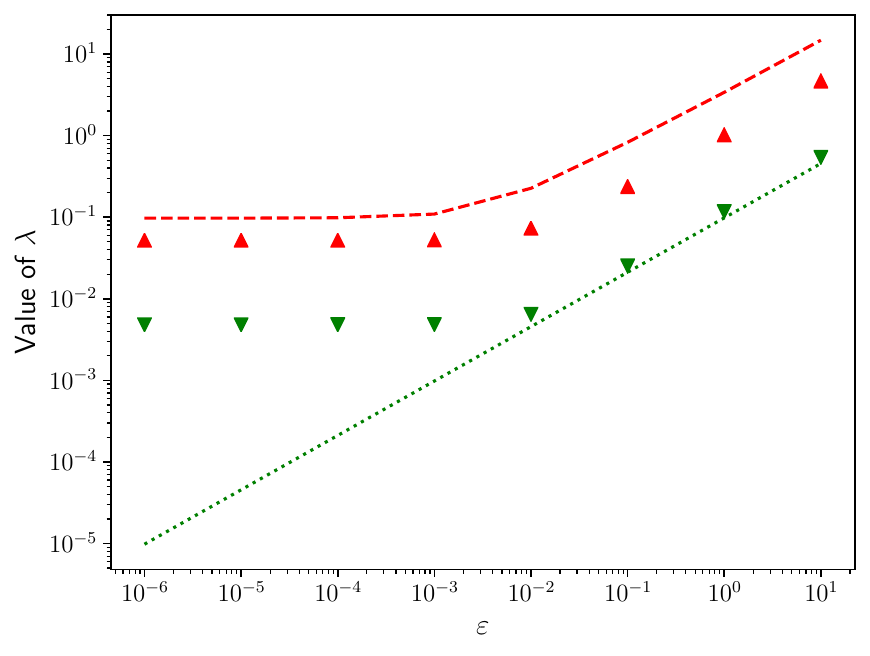}
    \caption{MS with low regularity - MINI}
    \label{fig:Mbelenki_mini}
  \end{subfigure}
  \hfill
  \begin{subfigure}[b]{\subfigwidthdoublecolumn}
    \centering
    \includegraphics[width=\textwidth]{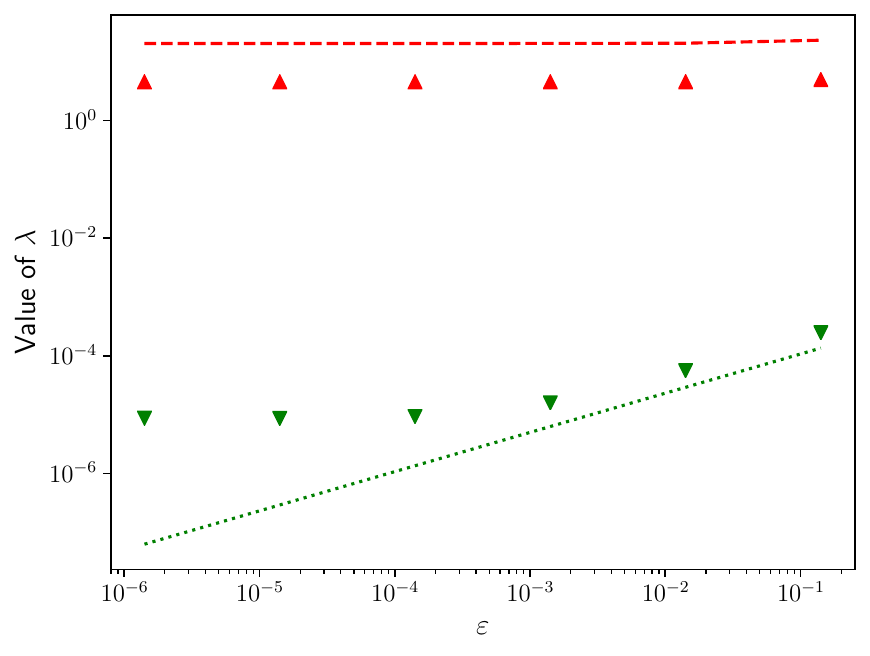}
    \caption{\experimentname{} - MINI}
    \label{fig:Marolla_mini}
  \end{subfigure}
  \caption{Eigenvalues of $M^{-1}S$ and their dependency on the
    regularization parameter $\varepsilon$ for the experiments using
    the manufactured solution (MS) with low regularity and Haut
    Glacier d'Arolla (\experimentname{}). The simulations are
    performed using the Newton method ($\gamma = 1$) and computations
    are made at the final non-linear iteration. The theoretical lower
    and upper eigenvalue bounds,
    \mbox{$c_0^2 \varepsilon^{2-\ppow} \leq \lambda \leq \frac{3}{\nu_0}
    (\varepsilon^2 +
    \|\mathbf{Du}_h^k\|^2_\infty)^{\frac{\ppow-2}{2}}$}, are
    shown as dotted green and dashed red lines, respectively
    (\cref{prop:M_bounds} with $\gamma = 1$ and $\ppow = 4/3$).
    Computed smallest non-zero, $\lambda_{min}$, and largest,
    $\lambda_{max}$, eigenvalues are shown as green down and red up triangles,
    respectively.}
  \label{fig:Mgraphs}
\end{figure*}

\subsection{Results for $\tilde{S}=M_\nu$}
\label{sec:results_Mnu}
As expected from theory, the eigenvalues are independent of $\varepsilon$ for the choice $\tilde{S}=M_\nu$, see \cref{fig:Mnugraphs}. Theory and experiments agree very well. Using $P2P1$ elements the ratio $\lambda_{max}/\lambda_{min}$  is smaller than $10^1$ for both the manufactured problem and Haut Glacier d'Arolla, which is beneficial for linear solvers. Using MINI elements for the glacier simulation does however result in a fairly large ratio $\lambda_{max}/\lambda_{min}$. This has to do with the value of $c_0$, which is affected by e.g., mesh quality, and will be explored more in the next section.

\begin{figure*}[h!]
     \centering
     \begin{subfigure}[b]{\subfigwidthdoublecolumn}
       \centering
\includegraphics[width=\textwidth]{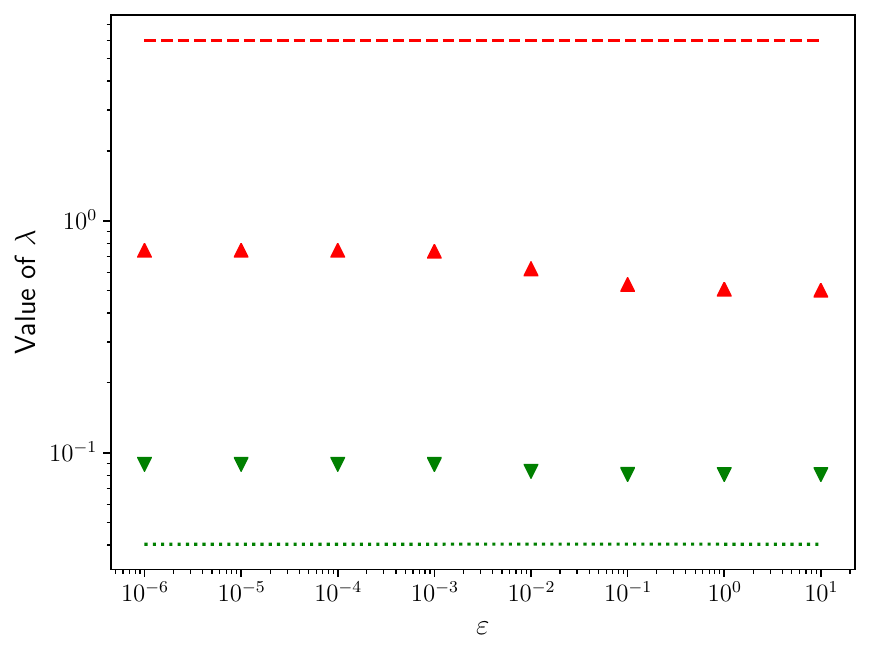}
       \caption{MS with low regularity - $P2P1$}
       \label{fig:Mnubelenki}
     \end{subfigure}
     \hfill
     \begin{subfigure}[b]{\subfigwidthdoublecolumn}
       \centering
       \includegraphics[width=\textwidth]{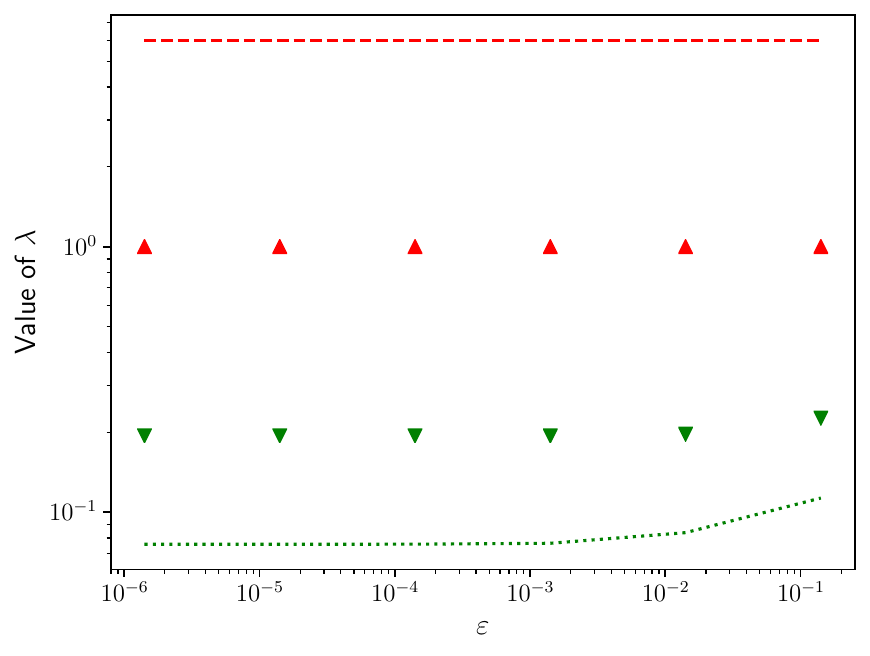}
       \caption{\experimentname{} - $P2P1$}
       \label{fig:Mnuarolla}
     \end{subfigure}
     \begin{subfigure}[b]{\subfigwidthdoublecolumn}
       \centering
       \includegraphics[width=\textwidth]{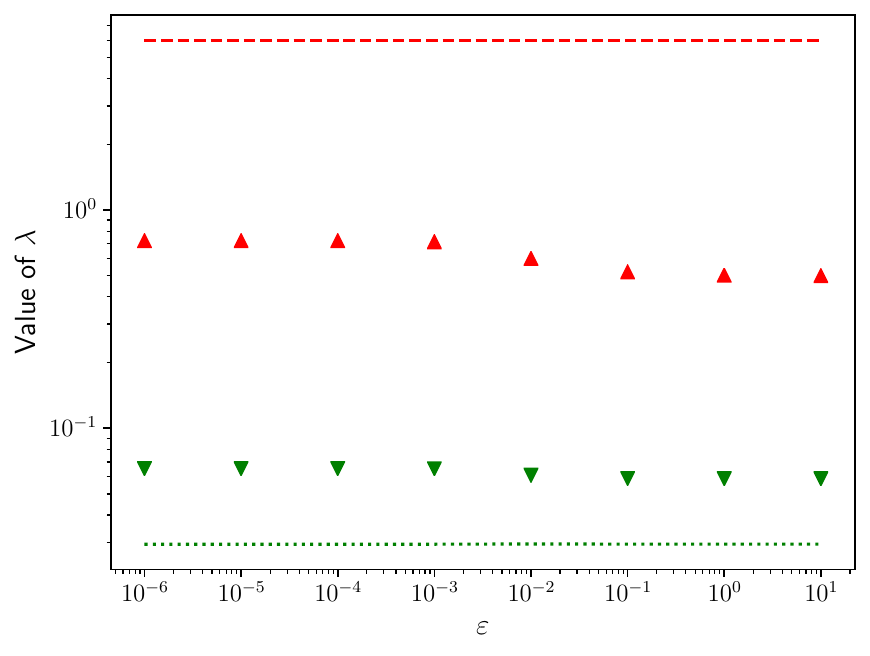}
       \caption{MS with low regularity - MINI}
       \label{fig:Mnubelenki_mini}
     \end{subfigure}
     \hfill
     \begin{subfigure}[b]{\subfigwidthdoublecolumn}
       \centering
\includegraphics[width=\textwidth]{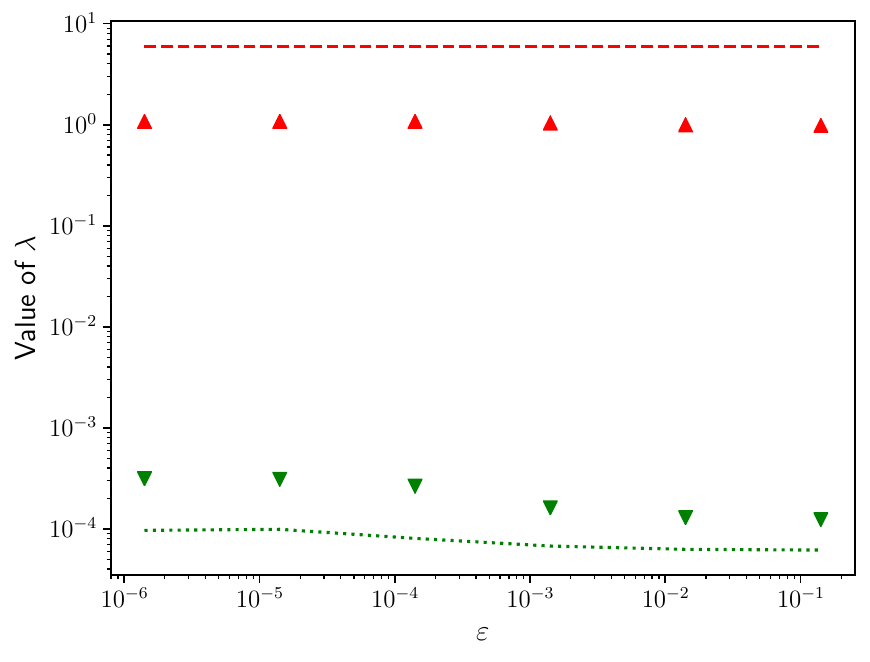}
       \caption{\experimentname{} - MINI}
       \label{fig:Mnuarolla_mini}
     \end{subfigure}
     \caption{Eigenvalues of $M_{\nu}^{-1}S$ and their dependency on
       the regularization parameter $\varepsilon$ for the experiments
       using the manufactured solution (MS) with low regularity and
       Haut Glacier d'Arolla (\experimentname{}). Colors and markers
       as in \cref{fig:Mgraphs}, except for the theoretical eigenvalue
       bounds which for $M_{\nu}^{-1}S$ are
       \mbox{$c_\nu^2 \leq \lambda \leq 6$} (\cref{prop:Mnu_bounds} with
       $\gamma = 1$ and $\ppow = 4/3$). }
     \label{fig:Mnugraphs}
\end{figure*}

\subsection{Mesh quality and MINI elements}
\label{sec:mesh_quality}
To further investigate how the smallest eigenvalues using MINI elements are affected by mesh quality, we below present results that show how the \emph{inf-sup}
constant $c_{\nu}$ is affected by the mesh quality.

\cref{fig:Mnu_mesh} shows the eigenvalue bounds for both experiments using the $P2P1$ element indicating that these are independent of the mesh sizes $nx$ and $lc$. For the MS experiment using MINI elements give very similar results. However, for the unstructured mesh used in \experimentname{}, the minimum eigenvalue $\lambda_{min}$ can be seen to decrease with a finer mesh sizes, see \cref{fig:Mnuarolla_mesh_mini}.

A reasonable explanation for this behavior is the locally reduced mesh quality that occurs when using an finer unstructured triangulation around the cusp-like geometry at the head of glacier (see \cref{fig:Arolla}). Compared to $P2P1$, using MINI with low-quality elements has a more significant impact on the \emph{inf-sup} constant $c_0$, which auxiliary \emph{inf-sup} constant $c_{\nu}$ depends on. The degradation of $\lambda_{min}$ shown in \cref{fig:Mnuarolla_mesh_mini} is most likely the result of the lower stability properties inherent in the MINI element. To support this view, we perform a simulation using a vertically extruded mesh (extrusion in 7 layers from bed to surface): such a mesh does not resolve the cusp-like feature as well, but results in elements of better quality (larger ratio of minimum/maximum element angles). Such extruded meshes are very common in large-scale ice-sheet models and are therefore of interest. \cref{fig:Mnuarolla_extruded_mini} shows the independence of the eigenvalues of the regularization parameter $\varepsilon$ for a simulation using an extruded mesh consisting of 1792 elements, approximately equal the amount of elements as the unstructured mesh with $lc=32$.

The numerical results presented above suggest that the \emph{inf-sup} constant $c_{\nu}$ in practice is very weakly dependent on, if not nearly independent of, the regularization parameter $\varepsilon$, but is directly connected through the dependence of $c_{\nu}$ to the regular \emph{inf-sup} constant $c_0$ (see \cref{tab:arolla_mini_infsup}). If so, the quality of $M_{\nu}$ as a preconditioner would be related to the stability qualities of a specific \emph{inf-sup} stable element, e.g., $P2P1$ or MINI.

\begin{figure*}[h!]
  \centering
  \begin{subfigure}[b]{\subfigwidthdoublecolumn}
    \centering
    \includegraphics[width=\textwidth]{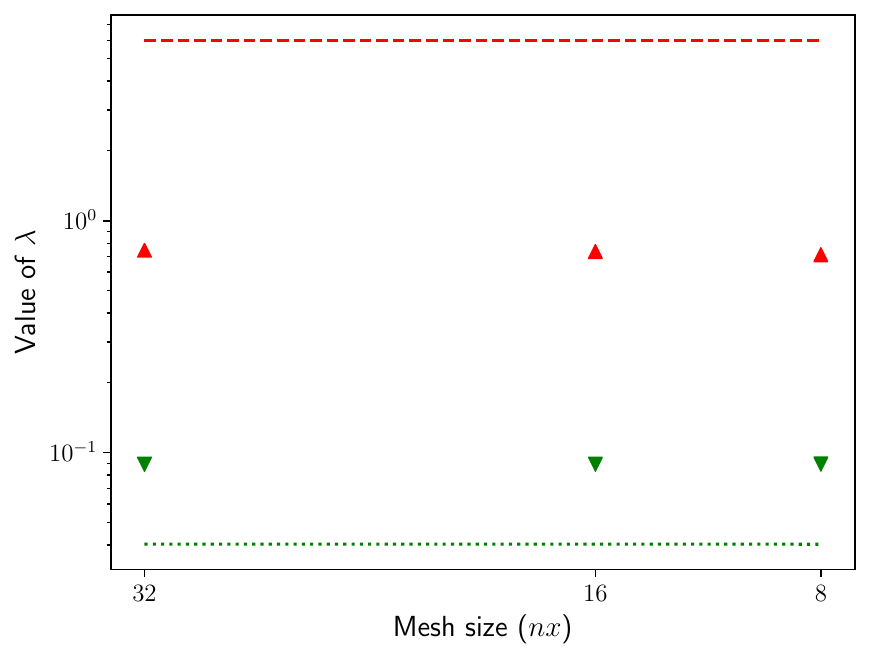}
    \caption{MS with low regularity - $P2P1$}
    \label{fig:Mnubelenki_mesh_p2p1}
  \end{subfigure}
  \hfill
  \begin{subfigure}[b]{\subfigwidthdoublecolumn}
    \centering
    \includegraphics[width=\textwidth]{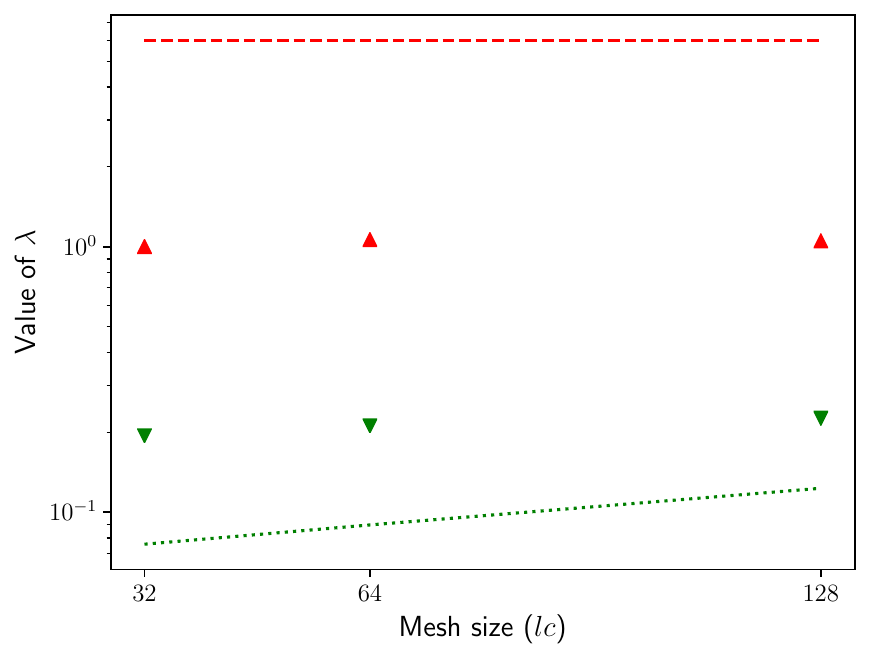}
    \caption{\experimentname{} - $P2P1$}
    \label{fig:Mnuarolla_mesh_p2p1}
  \end{subfigure}
  \caption{Eigenvalues of $M_{\nu}^{-1}S$ and their dependency on mesh
    size $nx$ and $lc$ for the experiments using the manufactured
    solution (MS) with low regularity and Haut Glacier d'Arolla
    (\experimentname{}), respectively. Mesh size is finer to the left. Colors and markers as in
    \cref{fig:Mnugraphs}.}
\label{fig:Mnu_mesh}
\end{figure*}

\begin{figure*}[h!]
  \centering
  \begin{subfigure}[b]{\subfigwidthdoublecolumn}
    \centering
    \includegraphics[width=\textwidth]{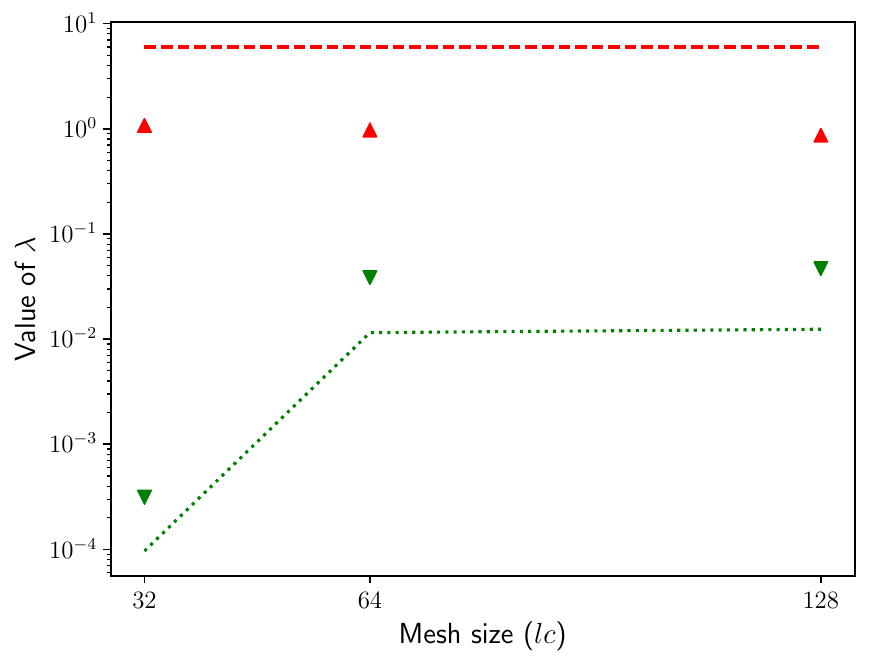}
    \caption{\experimentname{} - MINI on unstructured mesh}
    \label{fig:Mnuarolla_mesh_mini}
  \end{subfigure}
  \hfill
  \begin{subfigure}[b]{\subfigwidthdoublecolumn}
    \includegraphics[width=\textwidth]{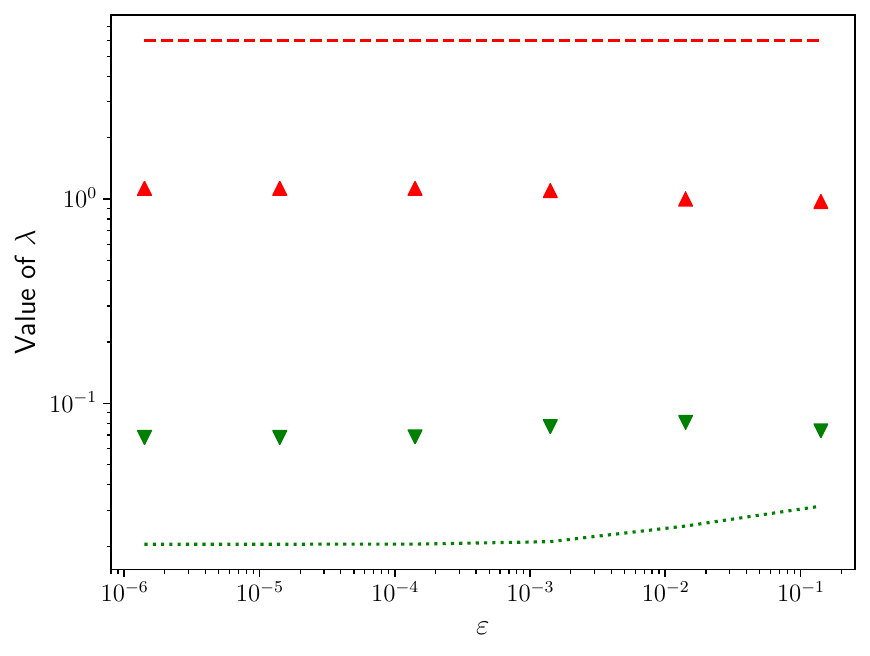}
    \caption{\experimentname{} - MINI on extruded mesh}
    \label{fig:Mnuarolla_extruded_mini}
  \end{subfigure}
    \begin{subfigure}[b]{\subfigwidthdoublecolumn}
    \centering
    \includegraphics[width=\textwidth]{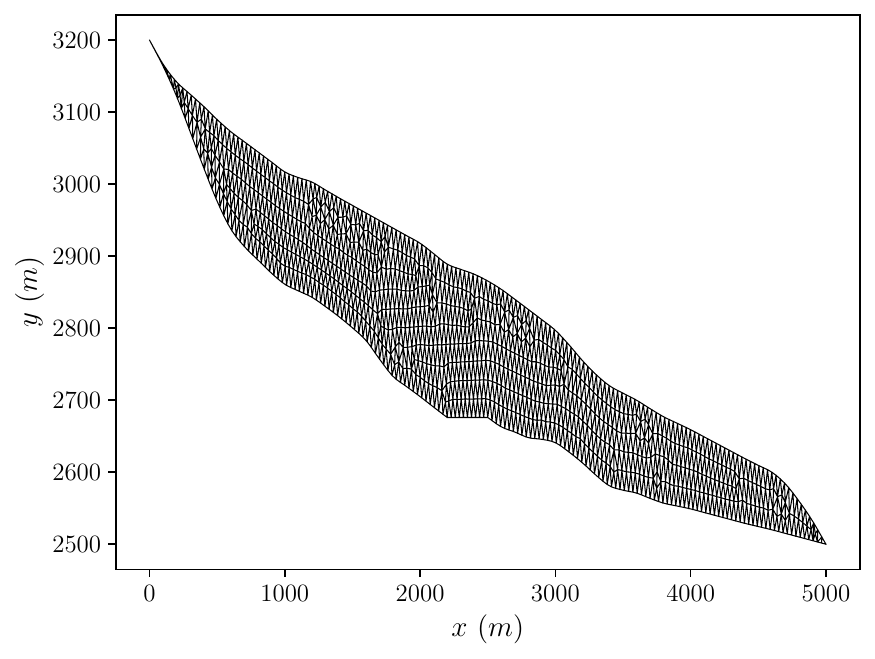}
    \caption{\experimentname{} - unstructured mesh}
    \label{fig:arolla_unstruct_mesh}
  \end{subfigure}
  \hfill
  \begin{subfigure}[b]{\subfigwidthdoublecolumn}
    \includegraphics[width=\textwidth]{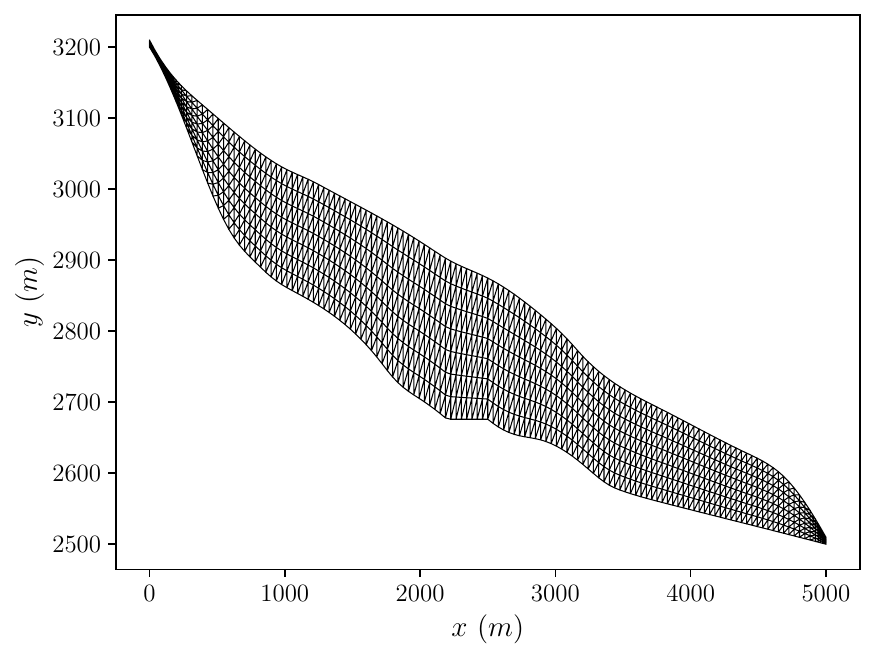}
    \caption{\experimentname{} - extruded mesh}
    \label{fig:arolla_extruded_mesh}
  \end{subfigure}

  \centering
  \caption{Eigenvalues of $M_{\nu}^{-1}S$ for the Haut Glacier
    d'Arolla (\experimentname{}) experiment. Top panel (colors and markers as in
    \cref{fig:Mnugraphs}): (a) shows the dependency of the eigenvalues
    on the mesh size $lc$ when using an unstructured mesh. The finest
    mesh size is $lc = 32$. (b) shows the dependency of the
    eigenvalues on the regularization parameter $\varepsilon$ using an
    extruded mesh with similar mesh size to $lc = 32$ (total number of
    elements in the triangulations are approximately equal). The
    bottom panel shows the unstructured (c) and extruded (d) meshes, respectively.}
  \label{fig:Mnu_extruded}
\end{figure*}

\section{Summary and Conclusion}
\label{sec:summary}
In this study we consider Schur-block preconditioners for the discretized $\ppow$-Stokes equations for fluids with shear-dependent viscosity. In particular, we focus on the regularized constitutive equation for the power-law fluid ice, which has no explicit lower bound for the viscosity and an upper bound inversely proportional to the regularization parameter $\varepsilon$. Based on previous results in \citet{GrinevichOlshanskii2009,He_etal2015} we adapt the theory to the considered power-law fluid and derive bounds of the eigenvalues of the preconditioned Schur block using either the mass matrix, $M$, or the viscosity-scaled mass matrix, $M_{\nu}$. Both the Newton and the Picard method are considered. For $M$ the lower bound for the eigenvalues depends directly on $\varepsilon$ and the upper bound on the maximum strain rate solved for in the previous non-linear iteration. For $M_{\nu}$ the lower bound depends only weakly and indirectly, as shown \citet{GrinevichOlshanskii2009}, on $\varepsilon$ while the upper bound is constant dependent of the given problem dimension for the Picard method and for Newton method additionally dependent on the value of $\ppow$. Hence, using $M_{\nu}$ results in a better clustering of eigenvalues for both the Picard and Newton method. Eigenvalues computed in numerical experiments show close agreement with the theoretical bounds presented. Furthermore, numerically computed eigenvalue bounds suggest that the lower eigenvalue bound for the $M_{\nu}$-preconditioned system is nearly independent of $\varepsilon$, further confirming the theoretical results from \citet{GrinevichOlshanskii2009}. The numerical experiments indicate that the lower bounds for $M_{\nu}$ depend more directly on the classical LBB \emph{inf-sup} connected to the domain and stability properties of a specific finite element. For MINI elements this suggests that the efficiency of the preconditioner depends on the quality of the mesh, which is particularly relevant to large-scale ice-sheet simulations, in which MINI elements are commonly used.
\\

\noindent {\bf Acknowledgments}\\
The authors would like to thank Prof. M. Olshanskii for kindly
answering our questions regarding the original theory presented in
{\protect\NoHyper\citet{GrinevichOlshanskii2009}\protect\endNoHyper}.
Funding for this research was provided by the Swedish e-Science
Research Centre (SeRC). The authors declare that they have no known
competing financial interests or personal relationships that could
have appeared to influence the work reported in this paper.

\bibliographystyle{abbrvnat}

\begin{thebibliography}{45}
\providecommand{\natexlab}[1]{#1}
\providecommand{\url}[1]{\texttt{#1}}
\expandafter\ifx\csname urlstyle\endcsname\relax
  \providecommand{\doi}[1]{doi: #1}\else
  \providecommand{\doi}{doi: \begingroup \urlstyle{rm}\Url}\fi

\bibitem[Aln{\ae}s et~al.(2015)Aln{\ae}s, Blechta, Hake, Johansson, Kehlet,
  Logg, Richardson, Ring, Rognes, and Wells]{Alnaes2015}
M.~Aln{\ae}s, J.~Blechta, J.~Hake, A.~Johansson, B.~Kehlet, A.~Logg,
  C.~Richardson, J.~Ring, M.~Rognes, and G.~Wells.
\newblock The {FEniCS} project version 1.5.
\newblock \emph{Archive of Numerical Software}, 3\penalty0 (100), 2015.
\newblock ISSN 2197-8263.
\newblock \doi{10.11588/ans.2015.100.20553}.
\newblock URL \url{http://journals.ub.uni-heidelberg.de/index.php/ans/article/
  view/20553}.

\bibitem[Arnold and Rognes(2009)]{ArnoldRognes2009}
D.~Arnold and M.~Rognes.
\newblock Stability of lagrange elements for the mixed laplacian.
\newblock \emph{Calcolo}, 46:\penalty0 245--260, 07 2009.
\newblock \doi{10.1007/s10092-009-0009-6}.

\bibitem[Babu{\v s}ka(1973)]{Babuska1973}
I.~Babu{\v s}ka.
\newblock The finite element method with {L}agrangian multipliers.
\newblock \emph{Numer. Math.}, 20\penalty0 (3):\penalty0 179--192, 1973.

\bibitem[Baiocchi et~al.(1993)Baiocchi, Brezzi, and Franca]{Baiocchi1993}
C.~Baiocchi, F.~Brezzi, and L.~P. Franca.
\newblock Virtual bubbles and {G}alerkin-least-squares type methods
  ({G}a.{L}.{S}.).
\newblock \emph{Comput. Methods in Appl. Mech. Eng.}, 105\penalty0
  (1):\penalty0 125 -- 141, 1993.
\newblock \doi{https://doi.org/10.1016/0045-7825(93)90119-I}.

\bibitem[Balay et~al.(1997)Balay, Gropp, McInnes, and Smith]{BalayEtAl1997}
S.~Balay, W.~D. Gropp, L.~C. McInnes, and B.~F. Smith.
\newblock Efficient management of parallelism in object oriented numerical
  software libraries.
\newblock In E.~Arge, A.~M. Bruaset, and H.~P. Langtangen, editors,
  \emph{Modern Software Tools in Scientific Computing}, pages 163--202.
  Birkh{\"{a}}user Press, 1997.

\bibitem[Balay et~al.(2016)Balay, Abhyankar, Adams, Brown, Brune, Buschelman,
  Dalcin, Eijkhout, Gropp, Kaushik, Knepley, McInnes, Rupp, Smith, Zampini,
  Zhang, and Zhang]{webpage:petsc}
S.~Balay, S.~Abhyankar, M.~F. Adams, J.~Brown, P.~Brune, K.~Buschelman,
  L.~Dalcin, V.~Eijkhout, W.~D. Gropp, D.~Kaushik, M.~G. Knepley, L.~C.
  McInnes, K.~Rupp, B.~F. Smith, S.~Zampini, H.~Zhang, and H.~Zhang.
\newblock {PETS}c {W}eb page.
\newblock \url{http://www.mcs.anl.gov/petsc}, 2016.
\newblock URL \url{http://www.mcs.anl.gov/petsc}.

\bibitem[Belenki et~al.(2012)Belenki, Berselli, Diening, and
  Růžička]{Belenki2012}
L.~Belenki, L.~C. Berselli, L.~Diening, and M.~Růžička.
\newblock On the finite element approximation of p-{S}tokes systems.
\newblock \emph{SIAM J. Numer. Anal.}, 50\penalty0 (2):\penalty0 373--397,
  2012.
\newblock \doi{10.1137/10080436X}.

\bibitem[Brezzi(1974)]{Brezzi1974}
F.~Brezzi.
\newblock On the {E}xistence, {U}niqueness and {A}pproximation of
  {S}addle-{P}oint {P}roblems {A}rising from {L}agrangian {M}ultipliers.
\newblock \emph{ESAIM-Math. Model. Num.}, 8\penalty0 (R2):\penalty0 129--151,
  1974.

\bibitem[Cuffey and Paterson(2010)]{Cuffey2010}
K.~M. Cuffey and W.~S.~B. Paterson.
\newblock \emph{The physics of glaciers}.
\newblock Academic Press, 2010.

\bibitem[Duval(1977)]{Duval1977}
P.~Duval.
\newblock The role of the water content on the creep rate of polycrystalline
  ice.
\newblock \emph{IAHS Publ.}, 118:\penalty0 29--33, 1977.

\bibitem[Duval et~al.(1983)Duval, Ashby, and Anderman]{Duval1983}
P.~Duval, M.~F. Ashby, and I.~Anderman.
\newblock Rate-controlling processes in the creep of polycrystalline ice.
\newblock \emph{J. Phys. Chem.}, 87\penalty0 (21):\penalty0 4066--4074, 1983.
\newblock \doi{10.1021/j100244a014}.
\newblock URL \url{http://dx.doi.org/10.1021/j100244a014}.

\bibitem[Elman et~al.(2005)Elman, Silvester, and Wathen]{Elman_etal2005}
H.~C. Elman, D.~J. Silvester, and A.~J. Wathen.
\newblock \emph{{Finite Elements and Fast Iterative Solvers with Applications
  in Incompressible Fluid Dynamics}}.
\newblock Numerical Mathematics and Scientific Computation. Oxford University
  Press, 2005.
\newblock ISBN 0198528671.

\bibitem[Fraters et~al.(2019)Fraters, Bangerth, Thieulot, Glerum, and
  Spakman]{Fraters_etal2019}
M.~R.~T. Fraters, W.~Bangerth, C.~Thieulot, A.~C. Glerum, and W.~Spakman.
\newblock {Efficient and practical Newton solvers for non-linear Stokes systems
  in geodynamic problems}.
\newblock \emph{Geophys. J. Int.}, 218\penalty0 (2):\penalty0 873--894, 04
  2019.
\newblock ISSN 0956-540X.
\newblock \doi{10.1093/gji/ggz183}.
\newblock URL \url{https://doi.org/10.1093/gji/ggz183}.

\bibitem[Gagliardini et~al.(2013{\natexlab{a}})Gagliardini, Zwinger,
  Gillet-Chaulet, Durand, Favier, de~Fleurian, Greve, Malinen, Mart{\'i}n,
  R{\aa}back, Ruokolainen, Sacchettini, Sch{\"a}fer, Seddik, and
  Thies]{Gagliardini2013}
O.~Gagliardini, T.~Zwinger, F.~Gillet-Chaulet, G.~Durand, L.~Favier,
  B.~de~Fleurian, R.~Greve, M.~Malinen, C.~Mart{\'i}n, P.~R{\aa}back,
  J.~Ruokolainen, M.~Sacchettini, M.~Sch{\"a}fer, H.~Seddik, and J.~Thies.
\newblock Capabilities and performance of {E}lmer/{I}ce, a new generation
  ice-sheet model.
\newblock \emph{Geosci. Model Dev.}, 6:\penalty0 1299--1318,
  2013{\natexlab{a}}.

\bibitem[Gagliardini et~al.(2013{\natexlab{b}})Gagliardini, Zwinger,
  Gillet-Chaulet, and et~al.]{ElmerDescrip}
O.~Gagliardini, T.~Zwinger, F.~Gillet-Chaulet, and et~al.
\newblock Capabilities and performance of {E}lmer/{I}ce, a new generation
  ice-sheet model.
\newblock \emph{Geosci. Model Dev.}, 6:\penalty0 1299--1318,
  2013{\natexlab{b}}.

\bibitem[Geuzaine and Remacle(2009)]{Geuzaine2009}
C.~Geuzaine and J.-F. Remacle.
\newblock Gmsh: A 3-{D} finite element mesh generator with built-in pre- and
  post-processing facilities.
\newblock \emph{Int. J. Numer. Meth. Eng.}, 79\penalty0 (11):\penalty0
  1309--1331, 2009.

\bibitem[Glen(1955)]{Glen1955}
J.~Glen.
\newblock The creep of polycrystalline ice.
\newblock \emph{Proc. Roy. Soc. Lond. A}, 228\penalty0 (1175):\penalty0
  519--538, 1955.

\bibitem[Grinevich and Olshanskii(2009)]{GrinevichOlshanskii2009}
P.~P. Grinevich and M.~A. Olshanskii.
\newblock An iterative method for the {S}tokes-type problem with variable
  viscosity.
\newblock \emph{SIAM J. Sci. Comput.}, 31\penalty0 (5):\penalty0 3959--3978,
  2009.
\newblock \doi{10.1137/08744803}.
\newblock URL \url{https://doi.org/10.1137/08744803}.

\bibitem[He et~al.(2015)He, Neytcheva, and Vuik]{He_etal2015}
X.~He, M.~Neytcheva, and C.~Vuik.
\newblock On preconditioning of incompressible non-{N}ewtonian flow problems.
\newblock \emph{J. Comput. Math.}, 33\penalty0 (1):\penalty0 33--58, 2015.
\newblock ISSN 1991-7139.
\newblock \doi{https://doi.org/10.4208/jcm.1407-m4486}.
\newblock URL \url{http://global-sci.org/intro/article_detail/jcm/9826.html}.

\bibitem[Henson and Yang(2002)]{HensonYang2002}
V.~E. Henson and U.~M. Yang.
\newblock Boomer{AMG}: a parallel algebraic multigrid solver and
  preconditioner.
\newblock \emph{Appl. Numer. Math.}, 41:\penalty0 155--177, 2002.

\bibitem[Hernandez et~al.(2005)Hernandez, Roman, and Vidal]{Hernandez_etal2005}
V.~Hernandez, J.~E. Roman, and V.~Vidal.
\newblock {SLEPc}: A scalable and flexible toolkit for the solution of
  eigenvalue problems.
\newblock \emph{ACM Trans. Math. Software}, 31\penalty0 (3):\penalty0 351--362,
  2005.

\bibitem[Hirn(2013)]{Hirnpowerlaw}
A.~Hirn.
\newblock Finite element approximation of singular power-law systems.
\newblock \emph{Math. Comp.}, 82\penalty0 (283):\penalty0 1247--1268, 2013.
\newblock ISSN 00255718, 10886842.
\newblock URL \url{http://www.jstor.org/stable/42002697}.

\bibitem[Isaac et~al.(2015)Isaac, Stadler, and Ghattas]{Isaac_etal2015}
T.~Isaac, G.~Stadler, and O.~Ghattas.
\newblock Solution of nonlinear stokes equations discretized by high-order
  finite elements on {N}onconforming and anisotropic meshes, with application
  to ice sheet dynamics.
\newblock \emph{SIAM J. Sci. Comput.}, 37\penalty0 (6):\penalty0 B804--B833,
  2015.
\newblock \doi{10.1137/140974407}.
\newblock URL \url{https://doi.org/10.1137/140974407}.

\bibitem[John(2002)]{John2002}
V.~John.
\newblock Slip with friction and penetration with resistance boundary
  conditions for the {N}avier–{S}tokes equations — numerical tests and
  aspects of the implementation.
\newblock \emph{J. Comput. Appl. Math.}, 147\penalty0 (2):\penalty0 287 -- 300,
  2002.
\newblock ISSN 0377-0427.
\newblock \doi{http://dx.doi.org/10.1016/S0377-0427(02)00437-5}.
\newblock URL \url{http://www.sciencedirect.com/science/article/pii/
  S0377042702004375}.

\bibitem[John(2016)]{John2016}
V.~John.
\newblock \emph{Finite element methods for incompressible flow problems},
  volume~51 of \emph{Springer Series in Computational Mathematics}.
\newblock Springer, 2016.
\newblock ISBN 3319457497 9783319457499.
\newblock \doi{10.1007/978-3-319-45750-5}.
\newblock URL \url{https://link.springer.com/book/10.1007%2F978-3-319-45750-5}.

\bibitem[Kaiser(2014)]{Kaiser2014}
K.~Kaiser.
\newblock Finite element methods for the incompressible {S}tokes equations with
  non-constant viscosity.
\newblock Master's thesis, Berlin Mathematical School, Freie Universität,
  Berlin, Germany, 01 2014.

\bibitem[Leng et~al.(2012)Leng, Ju, Gunzburger, Price, and Ringler]{Leng2012}
W.~Leng, L.~Ju, M.~Gunzburger, S.~Price, and T.~Ringler.
\newblock A parallel high-order accurate finite element nonlinear {S}tokes ice
  sheet model and benchmark experiments.
\newblock \emph{J. Geophys. Res.: Earth Surface}, 117\penalty0 (F1), 2012.
\newblock ISSN 2156-2202.
\newblock \doi{10.1029/2011JF001962}.
\newblock URL \url{http://dx.doi.org/10.1029/2011JF001962}.
\newblock F01001.

\bibitem[Logg et~al.(2012)Logg, Mardal, and Wells]{fenics:book}
A.~Logg, K.-A. Mardal, and G.~N. Wells.
\newblock \emph{Automated Solution of Differential Equations by the Finite
  Element Method}, volume~84 of \emph{Lecture Notes in Computational Science
  and Engineering}.
\newblock Springer, 2012.
\newblock \doi{10.1007/978-3-642-23099-8}.
\newblock URL \url{http://dx.doi.org/10.1007/978-3-642-23099-8}.

\bibitem[Malinen et~al.(2012)Malinen, Ruokolainen, Råback, Thies, and
  Zwinger]{Malinen_etal2012}
M.~Malinen, J.~Ruokolainen, P.~Råback, J.~Thies, and T.~Zwinger.
\newblock Parallel block preconditioning by using the solver of {E}lmer.
\newblock pages 545--547, 06 2012.
\newblock ISBN 978-3-642-36802-8.
\newblock \doi{10.1007/978-3-642-36803-5_43}.

\bibitem[May et~al.(2015)May, Brown, and {Le Pourhiet}]{May_etal2015}
D.~May, J.~Brown, and L.~{Le Pourhiet}.
\newblock A scalable, matrix-free multigrid preconditioner for finite element
  discretizations of heterogeneous stokes flow.
\newblock \emph{Comput. Methods Appl. Mech. Eng.}, 290:\penalty0 496--523,
  2015.
\newblock ISSN 0045-7825.
\newblock \doi{https://doi.org/10.1016/j.cma.2015.03.014}.
\newblock URL \url{https://www.sciencedirect.com/science/article/pii/
  S0045782515001243}.

\bibitem[Pattyn et~al.(2008)Pattyn, Perichon, Aschwanden, Breuer, {de Smedt},
  Gagliardini, Gudmundsson, Hindmarsh, Hubbard, Johnson, Kleiner, Konovalov,
  Martin, Payne, Pollard, Price, R\"uckamp, Saito, Sou\u{c}ek, Sugiyama, and
  Zwinger]{Pattyn2008}
F.~Pattyn, L.~Perichon, A.~Aschwanden, B.~Breuer, B.~{de Smedt},
  O.~Gagliardini, G.~H. Gudmundsson, R.~Hindmarsh, A.~Hubbard, J.~V. Johnson,
  T.~Kleiner, Y.~Konovalov, C.~Martin, A.~J. Payne, D.~Pollard, S.~Price,
  M.~R\"uckamp, F.~Saito, O.~Sou\u{c}ek, S.~Sugiyama, and T.~Zwinger.
\newblock Benchmark experiments for higher-order and full-{S}tokes ice sheet
  models ({ISMIP-HOM}).
\newblock \emph{Cryosphere}, 2:\penalty0 95--108, 2008.

\bibitem[P\"{o}rtner et~al.(2019)P\"{o}rtner, Roberts, and et~al.
  (eds.)]{IPCC2019}
H.-O. P\"{o}rtner, D.~Roberts, and et~al. (eds.).
\newblock \emph{{IPCC} {Special} {Report} on the {Ocean} and {Cryosphere} in a
  {Changing} {Climate}}.
\newblock Cambridge University Press, 2019.

\bibitem[Qin(1994)]{Qin1994}
J.~Qin.
\newblock \emph{On the convergence of some low order mixed finite elements for
  incompressible fluids}.
\newblock PhD thesis, Dept. of Mathematics, The Pennsylvania State University,
  01 1994.

\bibitem[R{\aa}back et~al.(2022)R{\aa}back, Malinen, Ruokalainen, Pursula, and
  Zwinger]{manual:Elmer}
P.~R{\aa}back, M.~Malinen, J.~Ruokalainen, A.~Pursula, and T.~Zwinger.
\newblock \emph{Elmer Models Manual}.
\newblock {CSC} -- {IT} {C}enter for {S}cience, Helsinki, Finland, 2022.

\bibitem[Rognes(2009)]{Rognes2009}
M.~E. Rognes.
\newblock {Automated stability condition tester (ASCoT)}, 2009.
\newblock URL \url{https://launchpad.net/ascot}.

\bibitem[Roman et~al.(2022)Roman, Campos, Dalcin, Romero, and
  Tomas]{slepc-users-manual}
J.~E. Roman, C.~Campos, L.~Dalcin, E.~Romero, and A.~Tomas.
\newblock {SLEPc} users manual.
\newblock Technical Report DSIC-II/24/02 - Revision 3.18, D. Sistemes
  Inform\`atics i Computaci\'o, Universitat Polit\`ecnica de Val\`encia, 2022.

\bibitem[R\"uckamp et~al.(2022)R\"uckamp, Kleiner, and
  Humbert]{Ruckamp_etal2022}
M.~R\"uckamp, T.~Kleiner, and A.~Humbert.
\newblock Comparison of ice dynamics using full-stokes and blatter--pattyn
  approximation: application to the northeast greenland ice stream.
\newblock \emph{The Cryosphere}, 16\penalty0 (5):\penalty0 1675--1696, 2022.
\newblock \doi{10.5194/tc-16-1675-2022}.
\newblock URL \url{https://tc.copernicus.org/articles/16/1675/2022/}.

\bibitem[Rudi et~al.(2017)Rudi, Stadler, and Ghattas]{Rudi_etal2017}
J.~Rudi, G.~Stadler, and O.~Ghattas.
\newblock Weighted {BFBT} preconditioner for {S}tokes flow problems with highly
  heterogeneous viscosity.
\newblock \emph{SIAM J. Sci. Comput.}, 39\penalty0 (5):\penalty0 S272--S297,
  2017.
\newblock \doi{10.1137/16M108450X}.
\newblock URL \url{https://doi.org/10.1137/16M108450X}.

\bibitem[Saad and Schultz(1986)]{SaadSchultz1986}
Y.~Saad and M.~H. Schultz.
\newblock {GMRES}: A generalized minimal residual algorithm for solving
  nonsymmetric linear systems.
\newblock \emph{SIAM J. Sci. Stat. Comput.}, 7\penalty0 (3):\penalty0
  856–869, jul 1986.
\newblock ISSN 0196-5204.

\bibitem[Schannwell et~al.(2020)Schannwell, Drews, Ehlers, Eisen, Mayer,
  Malinen, Smith, and Eisermann]{Schannwell_etal2020}
C.~Schannwell, R.~Drews, T.~A. Ehlers, O.~Eisen, C.~Mayer, M.~Malinen, E.~C.
  Smith, and H.~Eisermann.
\newblock Quantifying the effect of ocean bed properties on ice sheet geometry
  over 40\,000 years with a full-{S}tokes model.
\newblock \emph{The Cryosphere}, 14\penalty0 (11):\penalty0 3917--3934, 2020.
\newblock \doi{10.5194/tc-14-3917-2020}.
\newblock URL \url{https://tc.copernicus.org/articles/14/3917/2020/}.

\bibitem[Seddik et~al.(2012)Seddik, Greve, Zwinger, Gillet-Chaulet, and
  Gagliardini]{Seddik2012}
H.~Seddik, R.~Greve, T.~Zwinger, F.~Gillet-Chaulet, and O.~Gagliardini.
\newblock Simulations of the {G}reenland ice sheet 100 years into the future
  with the full {S}tokes model {E}lmer/{I}ce.
\newblock \emph{J. Glaciol.}, 58\penalty0 (209):\penalty0 427–440, 2012.
\newblock \doi{10.3189/2012JoG11J177}.

\bibitem[Seroussi et~al.(2020)Seroussi, Nowicki, and et~al.]{tc-14-3033-2020}
H.~Seroussi, S.~Nowicki, and et~al.
\newblock {ISMIP6} {A}ntarctica: a multi-model ensemble of the {A}ntarctic ice
  sheet evolution over the 21st century.
\newblock \emph{The Cryosphere}, 14\penalty0 (9):\penalty0 3033--3070, 2020.
\newblock \doi{10.5194/tc-14-3033-2020}.
\newblock URL \url{https://tc.copernicus.org/articles/14/3033/2020/}.

\bibitem[Shih et~al.(2021)Shih, Stadler, and Wechsung]{Shih_etal2021}
Y.-h. Shih, G.~Stadler, and F.~Wechsung.
\newblock Robust multigrid techniques for augmented {L}agrangian
  preconditioning of incompressible {S}tokes equations with extreme viscosity
  variations, 2021.
\newblock URL \url{https://arxiv.org/abs/2107.00820}.

\bibitem[Taylor and Hood(1974)]{TaylorHood1974}
C.~Taylor and P.~Hood.
\newblock {N}avier-{S}tokes equations using mixed interpolation.
\newblock \emph{Int. Symp. on Finite Element Methods in Flow Problems}, pages
  121--132, 1974.

\bibitem[Zhang et~al.(2011)Zhang, Ju, Gunzburger, Ringler, and
  Price]{Zhang2011}
H.~Zhang, L.~Ju, M.~Gunzburger, T.~Ringler, and S.~Price.
\newblock Coupled models and parallel simulations for three-dimensional
  full-{S}tokes ice sheet modeling.
\newblock \emph{Numer. Math. Theor. Meth. Appl.}, 4:\penalty0 359--381, 2011.

\end{thebibliography}

\appendix
\renewcommand{\thesection}{Appendix \Alph{section}}
\section{}
\Cref{tab:arolla_mini_infsup} show how the value of the \emph{inf-sup} constants $c_{\nu}$  and $c_0$ vary with mesh size ($lc$) for the \experimentname{} experiment using MINI elements. The strong correlations between $c_{\nu}$ and $c_0$ indicates that the dependence of $c_{\nu}$ on the regularization parameter $\varepsilon$ is as suggested in \citet{GrinevichOlshanskii2009} very weak.

\begin{table}[!h]
  \caption{Dependence of the \emph{inf-sup} constants $c_{\nu}$ and $c_0$ on mesh size for the \experimentname{} on unstructured mesh using MINI elements.}
  \label{tab:arolla_mini_infsup}
    \begin{tabular}[c]{||c | c c c||}
      \hline
      $lc$ & 32 & 64 & 128 \\ 
      \hline
      $c_0$ & 0.012 & 0.085 & 0.146 \\
      \hline
      $c_{\nu}$ & 0.010 & 0.107 & 0.111\\
      \hline
    \end{tabular}
\end{table}

\end{document}